\documentclass[a4paper,9pt]{article}
\usepackage{amsfonts}
\usepackage{latexsym,amssymb}
\usepackage{amsmath, amsbsy}
\usepackage{amsopn, amstext}
\usepackage{graphicx, color, epstopdf}
\usepackage{threeparttable}
\usepackage{accents}

\numberwithin{equation}{section}

\newtheorem{lemma}{Lemma}[section]

\newtheorem{theorem}{Theorem}[section]

\def\NC{\hbox{\rlap{\kern.24em\raise.1ex\hbox
                  {\vrule height1.3ex width.9pt}}C}}

\def\vec{\vec}

\def\ZZ{\mathbb{Z}}
\def\RR{\mathbb{R}}
\def\PP{\mathbb{P}}
\def\d{\mathrm{d}}
\def\e{\mathrm{e}}
\def\i{\mathrm{i}}
\def\tr{\mathsf{t}}
\def\n{\boldsymbol{n}}
\def\NN{\mathbb{N}}
  \def\OHm{\accentset{\circ}H_m}
  \def\OH{\accentset{\circ}H}
\newcommand{\bs}[1]{\boldsymbol{#1}}  
 \def \endproof {\vrule height5pt width 5pt depth 0pt}

\def\vec#1{{\boldmath
\mathchoice{\hbox{$\displaystyle#1$}}{\hbox{$\textstyle#1$}} Our
problem can be stated as the following two forms
{\hbox{$\scriptstyle#1$}}{\hbox{$\scriptscriptstyle#1$}}}}

\title{Spectral-Galerkin Approximation and Optimal Error Estimate  for Stokes Eigenvalue Problems %by the Stream Function Formulation
 in Polar Geometries\thanks{
  This work is supported in part by the National Natural Science Foundation of China grants No. 11661022, 91130014, 11471312, 91430216, 11471031, and U1530401; and by the US National Science Foundation grant DMS-1419040.}}
\author{Jing An\thanks{Beijing
Computational Science Research Center,  Beijing 100193, China ({\tt anjing@csrc.ac.cn}); and School of Mathematical Sciences, Guizhou Normal University, Guiyang 550025, China ({\tt aj154@163.com}).}
 \and Huiyuan Li\thanks{State Key Laboratory of Computer Science/Laboratory of Parallel Computing,  Institute of Software, Chinese Academy of Sciences, Beijing 100190, China ({\tt huiyuan@iscas.ac.cn}).}
\and Zhimin Zhang\thanks{Beijing
Computational Science Research Center,  Beijing 100193, China
{({\tt zmzhang@csrc.ac.cn}); and} Department of
Mathematics, Wayne State University, Detroit, MI 48202, USA
 ({\tt zzhang@math.wayne.edu}).}}
\date{}

\begin{document}

\maketitle

\begin{abstract}
In this paper we { propose and analyze spectral-Galerkin methods for the Stokes eigenvalue problem based on the} stream function formulation in polar geometries. { We first analyze the stream function} formulated  {fourth-order}  equation
%through the polar coordinate transformation
under the polar coordinates, { then we} derive the pole condition and reduce the problem on a {circular} disk to a sequence of equivalent  one-dimensional eigenvalue problems that can be solved in parallel.
{ The novelty of our approach lies in the construction} of  suitably weighted Sobolev spaces according to the pole conditions{, based on which, the optimal error estimate for approximated eigenvalue of each one dimensional problem can be obtained}.
Further, { we extend our method to the} non-separable Stokes eigenvalue {problem}
in an {elliptic domain and establish the optimal error bounds.}
   Finally, we provide some numerical experiments to validate {our} theoretical results and algorithms.


%\vskip 5pt \noindent {\bf AMS subject classification (2000):} 65F35, 15A12,
%15A18.
%\begin{AMS}
%65N35, 65N25, 35Q40
%\end{AMS}
%65N35      Numerical analysis on  Spectral, collocation and related methods
%65N25  	 Numerical analysis on Eigenvalue problems
%35P15   Partial differential equations; Estimation of eigenvalues, upper and lower bounds
%35Q30   Partial differential equations ;	Stokes and Navier-Stokes equations
%76M22     	Fluid mechanics  Spectral methods
%41A10  	Approximation by polynomials
%35Q40  	PDEs in connection with quantum mechanics

\end{abstract}

\section{Introduction}\label{int}
\label{Sec:Intro}
We consider in this paper  the Stokes eigenvalue problem which arises
{in stability analysis of the stationary solution} of the Navier-Stokes equations \cite{O76}:
\begin{align}
&-\Delta\bs u+\nabla p=\lambda\bs u,&&\text{in}~~\Omega ,\label{eqn:eigensys1} \\
&\nabla\cdot \bs u=0, &&\text{in}~~\Omega,\label{eqn:eigensys2}\\
&  \bs u=0,&&\text{on}~~\partial\Omega ,\label{eqn:eigensys3}
\end{align}
where $ \bs u=(u_1,u_2)$ is the flow velocity, $p$ is the pressure, $\Delta$ is the Laplacian operator, $\Omega$  is the flow domain and
$\partial\Omega$ denotes the boundary of the flow domain $\Omega$.

Let us introduce the stream function $\psi$ such that $\bs u=(\partial_y\psi,-\partial_x\psi)$.
Then we derive an alternative formulation
for  \eqref{eqn:eigensys1}-\eqref{eqn:eigensys3}:
\begin{align}
&-\Delta^2 \psi=\lambda \Delta\psi,&&\text{in}~~\Omega ,\label{eqn:eigensys4} \\
& \psi=\frac{\partial\psi}{\partial \n}=0,&&\text{on}~~\partial\Omega ,\label{eqn:eigensys5}
\end{align}
where $\n$ is the unit outward normal to the boundary $\partial\Omega$.  \eqref{eqn:eigensys4}  is also referred to as the 
biharmonic eigenvalue problem for plate buckling.
The  naturally equivalent weak form of \eqref{eqn:eigensys4}-\eqref{eqn:eigensys5} reads: Find $(\lambda, \psi)\in \RR\times H_0^2(\Omega)$ such that
\begin{align}
\mathcal{A}(\psi,\phi)
  = \lambda  \mathcal{B}(\psi,\phi), 
\quad \phi\in H^2_0(\Omega),\label{weak0}
\end{align}
where the bilinear forms  $\mathcal{A}$ and $\mathcal{B}$ are defined by
\begin{align*}
&\mathcal{A}(\psi,\phi) =(\Delta \psi, \Delta \phi) = \int_{\Omega} \Delta\psi \Delta \overline{\phi}\, \d x \d y,
\\
&\mathcal{B}(\psi,\phi) =(\nabla \psi, \nabla \phi) = \int_{\Omega} \nabla\psi \cdot \nabla \overline{\phi}\, \d x \d y.
\end{align*}

There are various numerical approaches to solving  \eqref{eqn:eigensys4}-\eqref{eqn:eigensys5}. Mixed finite element methods introduce  the auxiliary function $w=\Delta \psi$ to reduce the fourth-order equation 
to a saddle point problem   and then discretize the reduced  second order equations with ($C^0$-) continuous finite elements\cite{C78,MORR81,CL06,YJ13}. However, spurious solutions may occur in some situations. 
The conforming finite element methods including  Argyris  elements \cite{Argyris68} and the partition of unity finite elements \cite{Davis14}, require globally continuously differentiable finite element spaces, which are difficult to construct and implement.
The third type of approaches use non-conforming finite element methods, such as Adini elements \cite{AdiniClough68}, Morley elements \cite{Morley68,Rannacher79,Shi90} and the ordinary  $C^0$-interior penalty Galerkin method \cite{BreMonSun15}.   Their disadvantage lies in that such elements do not come in a natural hierarchy.
Both the  conforming and nonconforming finite element methods   are based on the naturally equivalent variational formulation \eqref{weak0}, and usually  involve low order polynomials and  guarantee only a low order of convergence.

In contrast, { it is observed in \cite{zhang} that the spectral method, whenever it is applicable, has tremendous advantage over the traditional $h$-version methods.}
  In particular, spectral  and spectral element methods { using} high order orthogonal polynomials for {fourth-order} equations result in an
exponential order of convergence for smooth solutions \cite{shen1997,BjTj97,Bialecki01,Guo05,YuGuo14,GuoYu16,ChenAnZhuang16}.
In analogy to the Argyris finite element methods, the conforming spectral element method requires globally continuously differentiable element spaces, which are extremely difficult to construct and implement on unstructured (triangular or quadrilateral) meshes.
This is exactly the reason why $C^1$-conforming spectral elements are rarely reported in literature except  those on rectangular meshes \cite{YuGuo14}. Hence, the spectral methods using globally smooth basis functions are naturally suitable choices  in practice
for %the fourth-order equation
 \eqref{weak0} on some fundamental regions including  rectangles, triangles and polar geometries.

To the best of our knowledge there are few reports on spectral-Galerkin approximation  for the Stokes eigenvalue problem by
the stream function formulation in polar geometries.
The polar transformation introduces  polar singularities  and  variable coefficients of the form $r^{\pm m}$ in polar coordinates \cite{shen1997,Bel01},
which involves  intricate pole conditions thus brings forth severe difficulties in both the design of  approximation schemes  and the corresponding error analysis. The aim of the current paper is  to  propose {and analyze} an efficient spectral-Galerkin approximation for the stream function formulation of the Stokes eigenvalue {problem in polar geometries}. As the first step,  we {use the} separation of variables in polar coordinates {to reduce} the original  problem in the unit disk to { equivalent infinite} sequence of one-dimensional eigenvalue problems which can be solved individually in parallel.
{Rigorous}  pole conditions involved are prerequisite for  the equivalence of the original problem and {the sequence of} the one-dimensional eigenvalue problems,  and thus play a fundamental role in our further study.  It is worthy to note, however,  that  the pole conditions derived for the {fourth-order} source problems in open literature  (such as \cite{shen1997,Bel01})  are
inadequate for our eigenvalue problems  since they would inevitably  induce improper/spurious computational results.

{ Based on the pole condition, suitable  approximation spaces are  introduced  and spectral-Galerkin schemes} are proposed.  A rigorous analysis on the optimal error estimate in certain properly introduced weighted Sobolev spaces   is made for each one dimensional eigenvalue problem  by using the minimax principle. { Finally, we extend our spectral-Galerkin method} to solving the  stream function formulation of the Stokes eigenvalue { problem in an elliptic} region.  Owing to its {non-separable property}, this problem is actually another challenge both in computation and analysis.    A brief explanation on the implementation of the approximation scheme  is first given, and an optimal error estimate is then presented in {the} Cartesian coordinates {under} the framework of  Bab\v{u}ska and Osborn \cite{babuska1991}.
%Finally, we provide some numerical experiments to validate the theoretical results and algorithms.

The rest of this paper is organized as follows. In the next section, dimension reduction scheme of the Stokes eigenvalue problem  is presented.
In \S3, we derive the weak formulation and prove the error estimation for a sequence of equivalent one-dimensional eigenvalue problems.
Also, we describe the details for an efficient implementation of the algorithm.
In \S4, we extend our algorithm to the case of elliptic region.
We present several numerical experiments in \S5 to demonstrate the accuracy and efficiency of our method.
Finally, in \S6 we give some concluding remarks.

\section{Dimensionality reduction and pole conditions}\label{Sec:Dimen}

Before coming to the main body of this section, we would like to introduce some notations and conventions which will be used throughout the paper.
Let $\omega$ be a generic positive weight function on a bounded domain $\Omega$, which is not necessarily in $L^1(\Omega)$. Denote by $(u,v)_{\omega,\Omega}$ the inner product of  $L^2_{\omega}(\Omega)$ whose norm is denoted by 
$\|\cdot\|_{\omega,\Omega}$. We use $H^m_{\omega}(\Omega)$ and $H^{m}_{0,\omega}(\Omega)$ to denote
the usual weighted Sobolev spaces, whose norm is denoted by
$\|\cdot\|_{m,\omega,\Omega}$. In cases where no confusion would arise, $\omega$ (if $\omega=1$) and $\Omega$
 may be dropped from the notation.
 Let $\NN_0$ (resp. $\ZZ$) be the collection of nonnegative integers (resp. integers).
For $N\in \NN_0$, we denote  by $\PP_N(\Omega)$ the collection of all  algebraic polynomials on $\Omega$ with the total degree no greater than $N$.
 We denote by $c$ a generic positive constant independent of any function and of any discretization parameters. We use the expression $A\lesssim B$ to mean that $A\le c B$.

In the current section, %and the forthcoming two sections,  
we restrict our attention to the unit disk $\Omega=D:=\{(x,y)\in \RR^2: x^2+y^2<1 \}$. 
We shall  employ a classical technique, separation of variables, to reduce the problem to a sequence of equivalent one-dimensional problems.

Throughout this paper, we shall use the polar coordinates $(r,\theta)$ for points in the disk $D$ such that $(x,y) = (r\cos \theta,r\sin \theta)$. We associate any function $u(x,y)$ in Cartesian coordinates with its  partner $\tilde u(r,\theta)=u(r\cos\theta,r\sin\theta)$
in polar coordinates.
If no confusion would arise, we shall 
use the same notation $u$ for $u(x,y)$ and $\tilde u(r,\theta)$.
We now  recall that, under the polar coordinates, 
\begin{align}
\Delta= \frac{1}{r} \frac{\partial}{\partial r}\left( r \frac{\partial}{\partial r}\right)+\frac{1}{r^2}
\frac{\partial^2}{\partial \theta^2},
\qquad \nabla = \big(\cos \theta \frac{\partial}{\partial r}- \frac{ \sin \theta}r \frac{\partial}{\partial \theta}, \sin \theta \frac{\partial}{\partial r}+ \frac{ \cos \theta}r \frac{\partial}{\partial \theta}   \big)^{\tr}.
 \label{a2.1}
\end{align}
Then the bilinear forms  $\mathcal{A}$ and $\mathcal{B}$  in \eqref{weak0} become
\begin{align*}
&\mathcal{A}(\psi, \phi)  =\int_{0}^1r  \d r \int_0^{2\pi}
\Big[\frac{\partial^2 \psi}{\partial r^2}  +\frac{1}{r} \frac{\partial \psi}{\partial r} +\frac{1}{r^2}
\frac{\partial^2 \psi}{\partial \theta^2}\Big]  \Big[\frac{\partial^2 \overline{\phi}}{\partial r^2}  +\frac{1}{r} \frac{\partial \overline{\phi}}{\partial r} +\frac{1}{r^2}
\frac{\partial^2 \overline{\phi}}{\partial \theta^2}\Big] \d \theta,
\\
 & \mathcal{B}(\psi, \phi)  =  \int_{0}^1r  \d r \int_0^{2\pi}
\Big[  \frac{\partial \psi}{\partial r}  \frac{\partial \overline{\phi}}{\partial r}   +   \frac{1}{r^2} \frac{\partial \psi}{\partial \theta}  \frac{\partial \overline{\phi}}{\partial \theta}  \Big] \d \theta.
\end{align*}

Denote $I=(0,1)$ and 
define the bilinear forms for functions $u,v$  on  $I$,
\begin{align*}
   &\mathcal{A}_m(u,v) = \int_0^1
\Big[ u'' +  \frac{u'}{r}- \frac{m^2 }{r^2} u\Big] \Big[ \overline{v}'' +   \frac{\overline{v}'}{r} -\frac{m^2}{r^2}\overline{v}  \Big] r \d r,
\\
   &\mathcal{B}_m(u,v) = \int_0^1
\Big(r u' \overline{v}'   +   \frac{m^2}{r} u \overline{v}  \Big) \d r.
\end{align*}
Further let us assume 
\begin{align}
\psi=\sum_{m\in \ZZ} \psi_m (r)\e^{\i m\theta},\qquad \phi = \sum_{m\in \ZZ}\phi_m (r) \e^{\i m\theta}. \label{expand}
\end{align}
By the orthogonality of  the Fourier system $\{\e^{\i m \theta} \}$, one finds that
\begin{align*}
\mathcal{A}(\psi, \phi) = 2\pi\, \sum_{m\in \ZZ} \mathcal{A}_m(\psi_m,\phi_m),
\qquad 
\mathcal{B}(\psi, \phi)  = 2\pi\,\sum_{m\in \ZZ}  \mathcal{B}_m(\psi_m,\phi_m).
\end{align*}

For the well-posedness of  $\mathcal{B}_m(\psi_m, \phi_m)$ and  $\mathcal{A}_m(\psi_m, \phi_m)$, the following pole conditions  for $\psi_m$  (and the same type of  pole conditions  for $\phi_m$) should be imposed,
\begin{align}
\label{pole}
 & m \psi_m(0) = 0,    \qquad 
 \lim_{r\to 0+}\Big[\psi_m'(r) - \frac{m^2}{r} \psi_m(r)\Big] = (1-m^2) \psi_m'(0) =0,
\end{align}
which can be further simplified into the following three categories,
\begin{align}
&(1). \quad \psi_m'(0)=0,  && m=0;\label{cond3}\\
&(2). \quad \psi_m(0)=0,   && |m|=1;\label{cond4}\\
&(3). \quad \psi_m(0)=\psi_m'(0)=0, && |m|\geq2.\label{cond5}
\end{align}
It is worthy to note that  our pole condition  \eqref{cond4} for  $|m|=1$
is a  revision of the pole condition   $\psi_m'(0)=\psi_m(0)=0$  in  (4.8) of \cite{shen1997}.
A concrete example to support the absence of $\psi_{\pm 1}'(0)=0$ reads,
$$
\psi = \psi_{\pm 1}(r) \e^{\pm \i \theta} \in H^2_0(D),  \qquad \psi_{\pm 1}(r) = (1-r)^2 r.
$$
Also, this absence of $\psi_{\pm 1}'(0)=0$ in \eqref{cond4} is also confirmed 
by \cite{Boyd89}.

The boundary conditions $\psi=\partial_{\n}\psi=0$ on $\partial D$ states $\psi_m'(1)=\psi_m(1)=0$ for all integer $m$.
Meanwhile,  $\mathcal{A}_m(\psi_m,\psi_m)=0$ together with $\psi_m'(1)=\psi_m(1)=0$ implies $\psi_m\equiv0$.
It is then easy to verify that   $\sqrt{\mathcal{A}_m(\psi_m,\psi_m)}$  (resp. $\sqrt{\mathcal{B}_m(\psi_m,\psi_m)}$\,)
induces a Sobolev norm for any function $\psi_m$  on $I$ which satisfies
the boundary condition $\psi_m'(1)=\psi_m(1)=0$ (resp. $\psi_m(1)=0$)  and the pole condition $
(1-m^2)\psi_m'(0)=m\psi_m(0)=0$ (resp. $m\psi_m(0)=0$).
 
We now  introduce two non-uniformly weighted Sobolev spaces on $I$,
\begin{align}
  \OHm^1 (I):= \big\{u:\  &\mathcal{B}_m(u,u) <\infty, 
\  mu(0) =   u(1)=0\big\},
\\
 \OHm^2(I):= \big\{u:\  &\mathcal{A}_m(u,u) <\infty, 
\  mu(0) = (1-m^2)u'(0)=   u(1)=u'(1)=0 \big\},
\end{align}
which are endowed with  energy norms
\begin{align}
\big\|u\big\|_{1,m,I}= \sqrt{ \mathcal{B}_m(u,u)}, \qquad 
\big\|u\big\|_{2,m,I}=  \sqrt{ \mathcal{A}_m(u,u)}.
\end{align}

In the sequel, \eqref{weak0}
is reduced to a system of  infinite  one-dimensional eigen problems: 
 to find $(\lambda_m,\psi_m)\in \RR\times\OHm^2(I)$ such that $\|\psi_m\|_{1,m,I}=1 $ and 
\begin{align}
\label{weak2}
\mathcal{A}_m(\psi_m,\phi_m) = \lambda_m  \mathcal{B}_m(\psi_m,\phi_m) , \quad \phi_m\in  \OHm^2(I), \quad m\in \ZZ.
\end{align}

We now conclude this section with the following  lemma on $\mathcal{A}_m(\cdot,\cdot)$ and $\mathcal{B}_m(\cdot,\cdot)$.

\begin{lemma} For $u,v\in \OHm^2(I)$,
\begin{align}
\label{buv0}
\begin{split}
\mathcal{B}_m(u,v)=&\,\int_{0}^1 \Big(u'\pm \frac{m}{r} u \Big)   \Big( \overline{v}' \pm \frac{m}{r}  \overline{v} \Big) r \d r,
 \end{split}
 \\
\label{auv0}
\begin{split}
\mathcal{A}_m(u,v)=&\,\int_{0}^1\Big[ r \Big(u'\mp\frac{m}{r} u\Big)'
\Big(\overline{v}'\mp\frac{m}{r} \overline{v}\Big)'+  \frac{(1\pm m)^2}{r}   \Big(u'\mp\frac{m}{r} u\Big)
 \Big(\overline{v}'\mp\frac{m}{r} \overline{v}\Big)  \Big]  \d r.
 \end{split}
\end{align}
\iffalse
Specifically, 
\begin{align}
\label{auv1}
&\mathcal{A}_0(u,v)=
\displaystyle \int_{0}^1\Big[ ru'' \overline{v}'' + \frac{1}{r} u' \overline{v}' \Big] \d r,
\\
\label{auv2}
&\mathcal{A}_{\pm 1}(u,v)=
\displaystyle \int_{0}^1\Big[ ru'' \overline{v}'' + \frac{3}{r}  
  \Big(u'\frac{u}{r} \Big)   \Big( \overline{v}'-\frac{ \overline{v}}{r} \Big) \Big] \d r, 
  \\
\label{auv3}
&\mathcal{A}_m(u,v)
=\displaystyle \int_{-1}^1\Big[ ru'' \overline{v}'' + \frac{2m^2+1}{r} u' \overline{v}' + \frac{m^4-4m^2}{r^3} u \overline{v}\Big] \d r
%\\
%&= \int_{-1}^1\Big[ ru'' \overline{v}'' + \frac{(m^2-1)^2}{r} u' \overline{v}' + \frac{m^2(4-m^2)}{r}    \Big(u'-\frac{u}{r} \Big)   \Big( \overline{v}'-\frac{ \overline{v}}{r} \Big) \Big] \d r
  .
\end{align}
\fi
\end{lemma}
\begin{proof}
By integration by parts and the pole condition \eqref{pole}, one verifies that
\begin{align*}
  \int_{0}^1 \Big(u'\pm &\frac{m}{r} u \Big)   \Big( \overline{v}' \pm \frac{m}{r}  \overline{v} \Big) r \d r
  =  \int_{0}^1 \Big (r  u'   \overline{v}'+  \frac{m^2}{r} u \overline{v}' \Big) \d r  
  \pm m \int_{0}^1   \big( u  \overline{v} \big)'  \d r
  \\
  =&\,  \int_{0}^1 \Big (r  u'   \overline{v}'+  \frac{m^2}{r} u \overline{v}' \Big) \d r ,
\end{align*}
which gives \eqref{buv0}.

Next, one readily checks that
\begin{align*}
%\label{L}
& u'' +  \frac{u'}{r}- \frac{m^2 u}{r^2}
 =  \Big(u'\mp\frac{m}{r} u\Big)' + \frac{1\pm m}{r} \Big(u'\mp\frac{m}{r} u\Big)
 %= \Big(\partial_r+\frac{1+m}{r}\Big) \Big(\partial_r-\frac{m}{r}\Big)u
 ,
 \\
& v'' +  \frac{v'}{r}- \frac{m^2 v}{r^2}
 =  \Big(v'\mp\frac{m}{r} u\Big)' + \frac{1\pm m}{r} \Big(v'\mp\frac{m}{r} v\Big)
% = \Big(\partial_r+\frac{1+m}{r}\Big) \Big(\partial_r-\frac{m}{r}\Big)v
.
% \notag
\end{align*}
As a result,
\begin{align*}
\mathcal{A}_m(u,v)=&\int_{0}^1\Big[ \Big(u'\mp\frac{m}{r} u\Big)'  +\frac{1\pm m}{r}\Big(u'\mp\frac{m}{r} u\Big) \Big]
\Big[\Big(\overline{v}'\mp\frac{m}{r} \overline{v}\Big)' + \frac{1\pm m}{r}\Big(\overline{v}'\mp\frac{m}{r} \overline{v}\Big) \Big] r\d r
\\
=&\int_{0}^1\Big[ r  \Big(u'\mp\frac{m}{r} u\Big)' \Big(\overline{v}'\mp\frac{m}{r} \overline{v}\Big)' + \frac{(1\pm m)^2}{r}  \Big(u'\mp\frac{m}{r} u\Big) \Big(\overline{v}'\mp\frac{m}{r} \overline{v}\Big)\Big]\d r 
\\&+\int_{0}^1  (1\pm m) \Big[\Big(u'\mp\frac{m}{r} u\Big)\Big(\overline{v}'\mp\frac{m}{r} \overline{v}\Big)\Big]'\d r
.\end{align*}
Meanwhile, the pole conditions \eqref{cond3}-\eqref{cond5}  states that  both $(1\pm m) \big(u'\mp\frac{m}{r} u\big)$
and  $(1\pm m) \big(v'\mp\frac{m}{r} v\big)$ vanish at the  two endpoints of $I$. Thus 
the last integral above is zero, and  \eqref{auv0} is now proved.
\iffalse
And \eqref{auv1}-\eqref{auv3} are corrected versions  of (4.10) in \cite{shen1997}.

\begin{align*}
  \mathcal{A}_m(u,v) =&\, \int_{0}^1\Big[  r u'' v'' + \frac{(1-m^2)^2}{r} u' v' 
  + \frac{m^4}{r} \Big( u' -\frac{u}{r}\Big)  \Big( v' -\frac{v}{r}\Big)\Big] \d r
  \\
   &\,+ (1-m^2) \int_{0}^1(u'v')' \d r + m^2\int_{0}^1 \Big[ u'' \Big( v' -\frac{v}{r}\Big) +v'' \Big( u' -\frac{u}{r}\Big) \Big]
   \d r 
\\   
   &\,+ m^2(1-m^2)\int_{0}^1 \frac{1}{r}\Big[ u' \Big( v' -\frac{v}{r}\Big) +v' \Big( u' -\frac{u}{r}\Big) \Big] \d r
\end{align*}
\begin{align*}
\int_{0}^1 \Big[ u'' \Big( v' -\frac{v}{r}\Big) +v'' \Big( u' -\frac{u}{r}\Big) \Big] \d r
\\
= \int_{0}^1 \Big[ \Big( u' -\frac{u}{r}\Big)' \Big( v' -\frac{v}{r}\Big) +\Big( v' -\frac{v}{r}\Big)' \Big( u' -\frac{u}{r}\Big) \Big] \d r
\\
+ \int_{0}^1\frac{1}{r} \Big[\Big( u' -\frac{u}{r}\Big) \Big( v' -\frac{v}{r}\Big) +\Big( v' -\frac{v}{r}\Big) \Big( u' -\frac{u}{r}\Big) \Big] \d r
\end{align*}

\begin{align*}
m^2(1-m^2)\int_{0}^1 \frac{1}{r}\Big[ u' \Big( v' -\frac{v}{r}\Big) + \Big( u' -\frac{u}{r}\Big) v'\Big] \d r
= m^2(1-m^2)\int_{0}^1 \frac{1}{r}\Big[ \Big( u' -\frac{u}{r}\Big)  \Big( v' -\frac{v}{r}\Big) +\Big( u' -\frac{u}{r}\Big) \Big( v' -\frac{v}{r}\Big)  \Big] \d r
\\
+m^2(1-m^2) \int_{0}^1\Big[ \frac{uv'+vu'}{r^2}  -\frac{2uv}{r^3} \Big] \d r
=m^2(1-m^2) \int_{0}^1 \frac{1}{r}\Big[ \Big( u' -\frac{u}{r}\Big)  \Big( v' -\frac{v}{r}\Big) +\Big( u' -\frac{u}{r}\Big) \Big( v' -\frac{v}{r}\Big)  \Big] \d r
\end{align*}

The proof is now  completed.
\fi
\end{proof}

\section{Spectral Galerkin approximation and  its error estimates}

Let $P_N(I)$ be the space of polynomials of degree less than or equal to $N$ on $I$, and setting $X_N^m=\PP_N(I)\cap  \OHm^2(I)$.
Then the spectral Galerkin approximation scheme to \eqref{weak2} is: Find $(\lambda_{mN},\psi_{mN})\in \RR\times X_N^m$ such that
$\|\psi_{mN}\|_{1,m}=1$ and 
\begin{align}
\mathcal{A}_m(\psi_{mN},v_N)=\lambda_{mN}\,\mathcal{B}_m(\psi_{mN},v_N),\qquad  \forall v_N\in  X_N^m.\label{weak3}
\end{align}
Due to the  symmetry  properties $\mathcal{A}_m=\mathcal{A}_{-m}$ and $\mathcal{B}_m=\mathcal{B}_{-m}$, we shall only consider $m\in \NN_0$
from now on in this section.
\subsection{Mini-max principle}

To give the error analysis, we will use extensively the minimax principle.
\begin{lemma} Let $\lambda_{m}^{l}$ denote the eigenvalues of (\ref{weak2}) and $V_{l}$ be any $l$-dimensional
subspace of $\OHm^2(I) $. Then, for $\lambda_m^1\leq \lambda_m^2 \leq \cdots \leq \lambda_m^l \leq \cdots $, there holds
\begin{eqnarray}\label{e3.3}
\lambda_m^l=\min_{V_l\subset \OHm^2(I) } \max_{v\in V_l}\frac{\mathcal{A}_m(v,v)}{\mathcal{B}_m(v,v)}.
\end{eqnarray}
\end{lemma}
\begin{proof}See Theorem 3.1 in \cite{mg2003}. \end{proof}

\begin{lemma}  Let $\lambda_m^i$ denote the eigenvalues of (\ref{weak2}) and
be arranged in an ascending order, and define
\begin{eqnarray*}
E_{i,j}= \hbox{\rm span}\left\{\psi_m^i,\cdots,\psi_m^j\right\},
\end{eqnarray*}
where $\psi_m^i$ is the eigenfunction corresponding to the eigenvalue $\lambda_m^i$. Then we have
\begin{eqnarray}\label{e3.4}
\lambda_m^l= \max_{v\in E_{k,l}}\frac{\mathcal{A}_m(v,v)}{\mathcal{B}_m(v,v)} && k\leq l,  \\\label{e3.5}
\lambda_m^l= \min_{v\in E_{l,m}}\frac{\mathcal{A}_m(v,v)}{\mathcal{B}_m(v,v)} && l\leq m .
\end{eqnarray}
\end{lemma}
\begin{proof} See Lemma 3.2 in \cite{mg2003}. \end{proof}

It is true that the minimax principle is also valid for the discrete formulation (\ref{weak3}) (see \cite{mg2003}).

\begin{lemma} {\it Let $\lambda_{mN}^l$ denote the eigenvalues of (\ref{weak3}), and $V_l$ be any $l$-dimensional subspace of $X_N^m$.
Then, for $\lambda_{mN}^1\leq \lambda_{mN}^2 \leq \cdots \leq \lambda_{mN}^l \leq \cdots$, there holds}
\begin{eqnarray}\label{e3.6}
\lambda_{mN}^l=\min_{V_l\subset X_N^m} \max_{v\in V_l}\frac{\mathcal{A}_m(v,v)}{\mathcal{B}_m(v,v)}.
\end{eqnarray}
\end{lemma}

Define the orthogonal projection $\Pi_{N}^{2,m}:\OHm^2(I)\mapsto X_N^m$ 
such that
 \begin{align}
 \label{def:Pi}
&\mathcal{A}_m(\psi_m-\Pi_{N}^{2,m}\psi_m,v)=0,\quad \forall v\in X_N^m.
\end{align}
\begin{theorem}
\label{th:diseigen}
 {\it Let $\lambda_{mN}^l$ be obtained by solving (\ref{weak3}) as an approximation of $\lambda_m^l$, an eigenvalue of (\ref{weak2}). Then, we have}
\begin{eqnarray}\label{e3.7}
0<\lambda_m^l\leq\lambda_{mN}^l\leq\lambda_m^l \max_{v\in E_{1,l}}\frac{\mathcal{B}_m(v,v)}{\mathcal{B}_m(\Pi_{N}^{2,m}v,\Pi_{N}^{2,m}v)}.
\end{eqnarray}
\end{theorem}
\begin{proof} According to the coerciveness of $\mathcal{A}_m(u,v)$ and $\mathcal{B}_m(u,v)$ we easily
 derive $\lambda_m^l>0$. Since $X_N^m \subset \OHm^2(I)$, from (\ref{e3.3}) and (\ref{e3.6}) we can obtain $\lambda_m^l\leq\lambda_{mN}^l$.
Let $\Pi_{N}^{2,m}E_{1,l}$ denote the space spanned by $\Pi_{N}^{2,m}\psi_m^1,\Pi_{N}^{2,m}\psi_m^2,\cdots,\Pi_{N}^{2,m}\psi_m^l$.  It is obvious that
$\Pi_{N}^{2,m}E_{1,l}$ is a $l$-dimensional subspace of $X_N^m$. From the minimax principle, we have
\begin{eqnarray*}
\lambda_{mN}^l\leq  \max_{v\in\Pi_{N}^{2,m}E_{1,l}}\frac{\mathcal{A}_m(v,v)}{\mathcal{B}_m(v,v)}=\max_{v\in E_{1,l}}\frac{\mathcal{A}_m(\Pi_{N}^{2,m}v,\Pi_{N}^{2,m}v)}{\mathcal{B}_m(\Pi_{N}^{2,m}v,\Pi_{N}^{2,m}v)}.
\end{eqnarray*}
Since
$
\mathcal{A}_m(v,v)=\mathcal{A}_m(\Pi_{N}^{2,m}v,\Pi_{N}^{2,m}v)+2a_m(v-\Pi_{N}^{2,m}v,\Pi_{N}^{2,m}v)+\mathcal{A}_m(v-\Pi_{N}^{2,m}v,v-\Pi_{N}^{2,m}v),
$
from  $\mathcal{A}_m(v-\Pi_{N}^{2,m}v,\Pi_{N}^{2,m}v)=0$ and the non-negativity of $a(v-\Pi_{N}^{2,m}v,v-\Pi_{N}^{2,m}v)$, we have
\begin{eqnarray*}
\mathcal{A}_m(\Pi_{N}^{2,m}v,\Pi_{N}^{2,m}v)\leq \mathcal{A}_m(v,v).
\end{eqnarray*}
Thus, we have
\begin{eqnarray*}
\lambda_{mN}^l &\leq& \max_{v\in E_{1,l}}\frac{\mathcal{A}_m(v,v)}{\mathcal{B}_m(\Pi_{N}^{2,m}v,\Pi_{N}^{2,m}v)}\\
&=&\max_{v\in E_{1,l}}\frac{\mathcal{A}_m(v,v)}{\mathcal{B}_m(v,v)}\frac{\mathcal{B}_m(v,v)}{\mathcal{B}_m(\Pi_{N}^{2,m}v,\Pi_{N}^{2,m}v)}\\
&\leq&\lambda_m^l \max_{v\in E_{1,l}}\frac{\mathcal{B}_m(v,v)}{\mathcal{B}_m(\Pi_{N}^{2,m}v,\Pi_{N}^{2,m}v)}.
\end{eqnarray*}
The proof of Theorem \ref{th:diseigen} is completed.  \end{proof}

\subsection{Error estimates}

Denote by $\omega^{\alpha,\beta}:=\omega^{\alpha,\beta}(r)=(1-r)^\alpha r^\beta$ the Jacobi weight function of index $(\alpha,\beta)$, which is not necessarily in  $L^1(I)$.
Define the $L^2$-orthogonal projection $\pi_N^{0,0}: L^2(I)\mapsto \PP_N(I) $ such that
\begin{align*}
   ( \pi_N^{0,0}u-u, v )_{I} =0, \qquad v\in \PP_N(I).
\end{align*} 
Further, for $k\ge 1$, define  recursively the $H^k$-orthogonal projections $\pi_N^{-k,-k}: H^k(I)\mapsto \PP_N(I)$ such that
$$
\big[\pi_N^{-k,-k}u \big](r)= \int_{0}^r \big[\pi_{N-1}^{1-k,1-k} u'\big](t) \d t + u(0).
$$

Next, for any nonnegative integers $s\ge k \ge 0$, define the Sobolev space
\begin{align*}
H^{s,k}(I) = \big\{u \in H^k(I):  \sum_{l=0}^{s} \|\partial_r^l u\|_{\omega^{\max(l-k,0),\max(l-k,0)},I} <\infty  \big\}.
\end{align*}
Now we have the  following error estimate on $\pi^{-k,-k}_N$.
\begin{lemma}[{\cite[Theorem 3.1.4]{Li2001}}]
$ \pi_N^{-k,-k}u$ is  a Legendre tau approximation of $u$ such that
\begin{align}
\label{pibnd}
& \partial_r^l\big[\pi_N^{-k,-k}u](0)=\partial_r^lu(0), \quad  \partial_r^l\big[\pi_N^{-k,-k}u](1)=\partial_r^lu(1),
\qquad 0\le l\le k-1,
\\
\label{tau}
& (\pi_N^{-k,-k}u-u, v) = 0, \qquad v\in \PP_{N-2k}.
\end{align}
Further suppose  $u\in H^{s,k}(I)$ with $s\ge k$. Then for $N\ge k$,
\begin{align}
\label{Pierror-1}
\begin{split}
   \| \partial_r^l( \pi_N^{-k,-k}u &  -u)\|_{\omega^{l-k,l-k},I}  \lesssim  N^{l-s}  \| \partial_r^su\|_{\omega^{s-k,s-k},I} , \quad 0\le l\le k\le s.
 \end{split}
\end{align} 
\end{lemma}

  \begin{theorem}
\label{lm:Pi}
Suppose $u \in \OHm^2(I)$ and $u'+\frac{m}{r}u\in H^{s-1,1}(I)$ with $s\ge 2$ and $m\in \NN_0$. Then for $N\ge m+3$,
\begin{align}
\label{Pierr}
\big\|\Pi_{N}^{2,m}u-u\big\|_{2,m,I} \lesssim \big( N+m \big) N^{1-s}\Big\|\partial_r^{s-1}\Big( \partial_r+\frac{m}{r}\Big) u\Big\|_{\omega^{s-2,s-2},I}.
\end{align}
\end{theorem}
\begin{proof} Define the differential operator $\mathcal{D}_m=\partial_r+\frac{m}{r} = \frac{1}{r^m}\partial_r ( r^m \cdot) $ 
and then set 
\begin{align*}
u_N(r) = -\frac{1}{r^{m}}\int_r^1 t^{m} \big[\pi^{-1,-1}_{N-1}\mathcal{D}_mu\big](t) \d t.
\end{align*}
We shall first prove $u_N\in X_N^m$.
By \eqref{tau}, we find that
\begin{align*}
\int_0^1& t^{m} \big[\pi^{-1,-1}_{N-1}\mathcal{D}_mu\big](t) \d t =
\int_0^1 t^{m}\big[ \mathcal{D}_mu\big](t) \d t =  \int_0^1 \partial_t\big[t^{m} u(t)\big]\d t = 0,\quad 
N\ge m+3, m\neq 0,
\end{align*}
where the last equality sign is derived from the boundary condition  $u(1)=0$.
Moreover,
\begin{align*}
&\partial_r\int_r^1 t^{m}  \big[\pi^{-1,-1}_{N-1}\mathcal{D}_mu\big](t)  \d t
=  -r^{m}\big[ \pi^{-1,-1}_{N-1}\mathcal{D}_mu\big](r).
\end{align*}
As a result,  $u_N\in \PP_N(I)$ and 
$$
u_N(0) = -\lim_{r\to 0} \frac{1}{r^{m}}\int_r^1 t^{m} \big[\pi^{-1,-1}_{N-1}\mathcal{D}_mu\big](t) \d t = 0, \qquad m\neq 0.
$$
Further, $u\in \OH^2_m(I)$ implies
\begin{align*}
\big[\mathcal{D}_mu\big](1)= 0,\quad m\in \ZZ; \qquad 
\big[\mathcal{D}_mu\big](0)= 0 , \quad m\neq 1,
\end{align*}
which, together with the property \eqref{pibnd} of $\pi^{-1,-1}_N$, gives 
\begin{align*}
\big[\pi^{-1,-1}_{N-1}\mathcal{D}_mu\big](1)= 0,\quad m\in \ZZ; \qquad 
\big[\pi^{-1,-1}_{N-1}\mathcal{D}_mu\big](0)= 0 , \quad m\neq 1.
\end{align*}
In the sequel,  we deduce that $u_N'(0)=0$ if $m\neq 1$ and $u_N(1)=u_N'(1)=0$.
In summary, we conclude that $u_N\in X_N^m$.

Next by \eqref{auv0} and \eqref{Pierror-1},  we have 
\begin{align*}
 \big\|u_N-u\big\|_{2,m}^2 =&\, \big\|\partial_r(\pi^{-1,-1}_{N-1}-I)\mathcal{D}_mu\big\|^2_{\omega^{0,1}}
 + (m-1)^2   \big\|(\pi^{-1,-1}_{N-1}-I)\mathcal{D}_mu\big\|^2_{\omega^{0,-1}}
 \\
 \le&\,  \big\|\partial_r(\pi^{-1,-1}_{N-1}-I)\mathcal{D}_mu\big\|^2_{\omega^{0,0}}
 + (m-1)^2   \big\|(\pi^{-1,-1}_{N-1}-I)\mathcal{D}_mu\big\|^2_{\omega^{-1,-1}}
 \\
 \lesssim &\, \big[N^{4-2s} +(m-1)^2  N^{2-2s}  \big]\,\big\|\partial_r^{s-1}\mathcal{D}_mu\big\|^2_{\omega^{s-2,s-2}}.
\end{align*}

Finally, \eqref{Pierr} is an immediate consequence of the projection  theorem,
\begin{align*}
  \big\|\Pi_{N}^{2,m}u-u\big\|_{2,m,I}  = \inf_{v\in X_N^m}  \big\|v-u\big\|_{2,m,I}
  \le  \big\|u_N-u\big\|_{2,m,I}.
\end{align*}
The proof is now completed.
\end{proof}

\begin{theorem}
Let $\lambda_{mN}^l$ is the $l$-th approximate eigenvalue of $\lambda_m^l$.\
If  $\{\psi_m^{i}\}_{i=1}^{l} \subset \OHm^2(I)\cap H^{s,2}(I)$ with $s\geq 2$, then we have
\begin{align*}
|\lambda_{mN}^l-\lambda_m^l|\lesssim ( N^2 + m^2 ) N^{2-2s} \max_{1\le i\le l }\Big\|\partial_r^{s-1}\Big(\partial_r+\frac{m}{r}\Big) \psi_m^i\Big\|^2_{\omega^{s-2,s-2},I}.
\end{align*}
\end{theorem}
\begin{proof}
For any $0\neq v\in E_{1,l}$,  it can be represented by $v=\sum_{i=1}^l\mu_i\psi_m^i$;  we then have
\begin{align*}
&\frac{\mathcal{B}_m(v,v)-\mathcal{B}_m(\Pi_{N}^{2,m}v,\Pi_{N}^{2,m}v)}{\mathcal{B}_m(v,v)}\leq \frac{2|\mathcal{B}_m(v,v-\Pi_{N}^{2,m}v)|}{\mathcal{B}_m(v,v)}\\
&\qquad \leq\frac{2\sum_{i,j=1}^{l}|\mu_i| |\mu_j||\mathcal{B}_m(\psi_m^i-\Pi_{N}^{2,m}\psi_m^i,\psi_m^j)|}{\sum_{i=1}^{l}|\mu_i|^2}\\
&\qquad \leq 2l\max_{i,j=1,\cdots,l}|\mathcal{B}_m(\psi_m^i-\Pi_{N}^{2,m}\psi_m^i,\psi_m^j)|
:= \varepsilon.
\end{align*}

Meanwhile, by the variational form \eqref{weak2}, the definition \eqref{def:Pi} 
of $\Pi_N^{2,m}$,  Cauchy-Schwarz inequality
 and Theorem \ref{lm:Pi},
we have
\begin{align*}
&|\mathcal{B}_m(\psi_m^i-\Pi_{N}^{2,m}\psi_m^i,\psi_m^j)|=\frac{1}{\lambda_m^j}|\lambda_m^jb_m(\psi_m^j,\psi_m^i-\Pi_{N}^{2,m}\psi_m^i)|\\
&=\frac{1}{\lambda_m^j}|\mathcal{A}_m(\psi_m^j,\psi_m^i-\Pi_{N}^{2,m}\psi_m^i)|
=\frac{1}{\lambda_m^j}|\mathcal{A}_m(\psi_m^j-\Pi_{N}^{2,m}\psi_m^j,\psi_m^i-\Pi_{N}^{2,m}\psi_m^i)|\\
&\leq\frac{1}{\lambda_m^j} \, \|\psi_m^j-\Pi_{N}^{2,m}\psi_m^j\|_{2,m,I}\, \|\psi_m^i-\Pi_{N}^{2,m}\psi_m^i\|_{2,m,I} 
\\
&\lesssim (N^2+m^4) N^{2-2s}  \|\partial_r^s\psi_m^j\|_{\omega^{s-2,s-2},I}\, \|\partial_r^s\psi_m^i\|_{\omega^{s-2,s-2},I}.
\end{align*}
As a result, we have the following estimate for $\varepsilon$,
\begin{align*}
  \varepsilon %\lesssim   (N^2+m^4) N^{2-2s} \max_{1\le i,j\le l}  \|\partial_r\psi_m^j\|_{\omega^{s-2,s-2},I}\, \|\partial_r\psi_m^i\|_{\omega^{s-2,s-2},I}
  \lesssim  ( N^2 + m^2 ) N^{2-2s} \max_{1\le i\le l }\Big\|\partial_r^{s-1}\Big(\partial_r+\frac{m}{r}\Big) \psi_m^i\Big\|^2_{\omega^{s-2,s-2},I}.
\end{align*}
For sufficiently large $N$, $\varepsilon<\frac12$. Thus
\begin{align*}
&0<\frac{\mathcal{B}_m(v,v)}{\mathcal{B}_m(\Pi_{N}^{2,m}v,\Pi_{N}^{2,m}v)} \leq\frac{1}{1-\varepsilon} \le 1+ 2\varepsilon,
\end{align*}
and we finally deduce from  Theorem \ref{th:diseigen} that
\begin{align*}
0<\lambda_{mN}^l-\lambda_m^l \le  2\lambda_m^l   \varepsilon \lesssim  
( N^2 + m^2 ) N^{2-2s} \max_{1\le i\le l }\Big\|\partial_r^{s-1}\Big(\partial_r+\frac{m}{r}\Big) \psi_m^i\Big\|^2_{\omega^{s-2,s-2},I}.
\end{align*}
The proof is now completed.
\end{proof}

\subsection{Implementations}
We describe in this section how to solve the problems \eqref{weak3} efficiently. To this end, we first construct a set of basis functions for $X_N^m$.
Let
\begin{align}
 \phi_i(r) =(1-r)^2r^2J^{2,1}_{i-4}(2r-1),\quad  i \ge 4, \label{basis1}
\end{align}
where $J_k^{\alpha,\beta}$ is the  Jacobi polynomial of degree $k$.

It is clear that
\begin{align*}
&X_N^m=\mathrm{span}\{\phi^m_i=\phi_i: 4\le i\le N \},\qquad m\ge 2
\\
&X_N^0=\mathrm{span}\{\phi^0_i=\phi_i: 4\le i\le N \}
\oplus \mathrm{span}\big\{\phi_{3}^0(r)=\tfrac{1}{4}(1-r)^2(2r+1)\big\},
\\
&X_N^1=\mathrm{span}\{\phi^1_i=\phi_i: 4\le i\le N \}
\oplus \mathrm{span}\big\{\phi_{3}^1(r)=\tfrac{1}{2}(1-r)^2r\big\}.
 \end{align*}
 
 Define $N_m=4$ if $m\ge2$ and $N_m=3$ otherwise. 
Our basis functions lead to the penta-diagonal  matrix $A^m=[\mathcal{A}_m(\phi^m_i,\phi^m_j)]_{N_m\le i,j\le N}$
 and the deca-diagonal mass matrix $B^m=[\mathcal{B}_m(\phi^m_i,\phi^m_j)]_{N_m\le i,j\le N}$
 instead of the hepta- and hendecagon-diagonal ones in \cite{shen1997}.
 
 \begin{lemma}
 \label{Expan}
 For $i\ge 4$,
 \begin{align}
 \label{phid2}
& \phi_i''(r) = \tfrac { (i-3) ( i-1 ) i}{2(2i-3)}  J^{0,1}_{i-2}(2r-1) 
 +   \tfrac {2 ( i-3 ) ^2 ( i-2 )  ( i-1
 ) }{ ( 2i-3 )  ( 2i-5 ) }  J^{0,1}_{i-3}(2r-1)
+ \tfrac { ( i-4 )  ( i-3 ) ^{2}}{2(2i-5)}J^{0,1}_{i-4}(2r-1),
\\
 \label{phid1}
&\phi_i'(r) = r\Big[\tfrac { ( i-3 ) i}{2(2i-3)} J^{0,1}_{i-2}(2r-1) 
 -\tfrac {3 ( i-3 )  ( i-2 ) }{ ( 2i-3 )  ( 2i-5 ) } J^{0,1}_{i-3}(2r-1)
-\tfrac { ( i-4 )  ( i-3 ) }{2(2i-5)}J^{0,1}_{i-4}(2r-1)\Big]
\\
\label{phid1r}
\begin{split}
&\quad=\tfrac{(i-3)i^2}{4(2i-3)(2i-1)}  J^{0,1}_{i-1}(2r-1)
+\tfrac{(i-1)(i-3)(2i^2-7i+2)}{2(2i-5)(2i-1)(2i-3)} J^{0,1}_{i-2}(2r-1)
-\tfrac{(i-2)(i-3)}{(2i-3)(2i-5)} J^{0,1}_{i-3}(2r-1)
\\
&\qquad
-\tfrac{(2i^2-9i+6)(i-3)^2}{2(2i-7)(2i-5)(2i-3)} J^{0,1}_{i-4}(2r-1)
-\tfrac{(i-3)(i-4)^2}{4(2i-5)(2i-7)} J^{0,1}_{i-5}(2r-1),
\end{split}
\\
 \label{phid0}
& \frac{\phi_i(r)}{r} =r\Big[ \tfrac {i-3}{2(2i-3)}J^{0,1}_{i-2}(2r-1) 
-\tfrac { 2 ( i-3 )  ( i-2 ) }{ ( 2i-3 )  ( 2i-5 ) }J^{0,1}_{i-3}(2r-1)
+\tfrac {i-3}{2(2i-5)} J^{0,1}_{i-4}(2r-1)\Big]
\\
\label{phid0r}
\begin{split}
&\quad =\tfrac{i(i-3)}{4(2i-1)(2i-3)} J^{0,1}_{i-1}(2r-1)
-\tfrac{(i-1)(i-3)}{(2i-1)(2i-3)(2i-5)} J^{0,1}_{i-2}(2r-1)
-\tfrac{(i-2)(i-3)}{2(2i-3)(2i-5)} J^{0,1}_{i-3}(2r-1)
\\
&\qquad 
+\tfrac{(i-3)^2}{(2i-7)(2i-3)(2i-5)} J^{0,1}_{i-4}(2r-1)
+\tfrac{(i-3)(i-4)}{4(2i-5)(2i-7)} J^{0,1}_{i-5}(2r-1),
\end{split}
 \end{align}
 and 
 \begin{align}
 \label{phi3d2}
 & \phi_3^0{}''(r) =  J^{0,1}_{1}(2r-1) + \tfrac{1}{2}  J^{0,1}_{0}(2r-1),\qquad \qquad 
 \phi_3^1{}''(r) =  J^{0,1}_{1}(2r-1),
\\
 \label{phi30}
 \begin{split}
  &\phi_3^0{}'(r) = \frac{r}{2} \big[ J^{0,1}_{1}(2r-1)-  J^{0,1}_{0}(2r-1) \big]
\\  
& \qquad  =\tfrac{3}{20}J^{0,1}_{2}(2r-1)+\tfrac{1}{10} J^{0,1}_{1}(2r-1) -\tfrac{1}{4} J^{0,1}_{0}(2r-1),
\end{split}
 \\
 \label{phi31}
  \begin{split}
&  \phi_3^1{}'(r) - \frac{ \phi_3^1(r)}{r} = \frac{r}{3}\big[  J^{0,1}_{1}(2r-1)- J^{0,1}_{0}(2r-1) \big] 
\\  
& \qquad =\tfrac{1}{10}J^{0,1}_{2}(2r-1)+\tfrac{1}{15} J^{0,1}_{1}(2r-1) -\tfrac{1}{6} J^{0,1}_{0}(2r-1).
\end{split}
 \end{align}
Thus 
 for $j\ge i$,
 \begin{align}
 \label{Aij}
 &A_{i,j}^{m} =\begin{cases}
\frac {  ( i-3 )  ( i-2 )  ( 3{i}^{4}+2{i}^{2}{m}^{2}+3{m}^{4}-24{i}^{3}-8i{m}^{2}+70{i}^{2}-88i+42 ) }{ 4 ( 2i-5 )  ( i-1 )  ( 2i-3 ) }
+(\frac{9}{16}\delta_{m,0}+\frac12\delta_{m,1})  \delta_{i,3},& j=i,\\[0.3em]
\frac { ( i-2 )  ( i-3 )  ( 4{i}^{4}-4{m}^{4}-24{i}^{3}+50{i}^{2}+10{m}^{2}-42
i+9 ) }{ 4 ( 2i-5 )  ( 2i-1 )  ( 2i-3 ) }
+  (\frac{3}{20}\delta_{m,0}+\frac14\delta_{m,1})  \delta_{i,3}, & j=i+1,\\[0.3em]
\frac { ( i-3 )  ( i-2+m ) 
 ( i-2-m )  ( i-m )  ( i+m ) }{ 8
 ( 2i-1 )  ( 2i-3 ) }+  \frac{3}{40}\delta_{m,0}  \delta_{i,3}, & j=i+2,\\[0.3em]
 0,  & j\ge i+3,
 \end{cases}
  \end{align}
  and
 \begin{align}
 \label{Bij}
 &B_{i,j}^{m} = \begin{cases}
  \frac {  ( i-3 )^2 ( i-2 )  ( 5i^2+3m^2-20i+8 ) }{ 16 ( 2i-7 )  ( 2i-1 )  ( 2i-3 )  ( 2i-5 ) }
  +  (\frac{3}{80}\delta_{m,0}+\frac{1}{60}\delta_{m,1})  \delta_{i,3}, & j=i,\\[0.3em]
\frac { ( i-2 )  ( i-3 )  ( 4i^4-24{i}^{3}+43{i}^{2}+6{m}^{2}-21i-26 ) }{16  ( 2i-7 )  ( 2i-5 )  ( 2i+1 )  ( 2i-1 )  ( 2i-3 ) }+  (\frac{1}{140}\delta_{m,0}+\frac{1}{84}\delta_{m,1})  \delta_{i,3},   & j=i+1,\\[0.3em]
 -\frac { ( i-1 ) ^{2} ( i-3 )  ( {i}^{2}+{m}^{2}-2i-4 ) }{ 8 ( 2i-5 )  ( 2i+1 )  ( 2i-1 )  ( 2i-3 ) } -  \frac{3}{560}\delta_{m,0}  \delta_{i,3},& j=i+2,\\[0.3em]
 -\frac { i ( i-3 )  ( 4 {i}^{4}-8{i}^{3}-13{i}^{2}+6{m}^{2}+17i-6 ) }{16 ( 2i-5 )  ( 2i-3 )  ( 2i-1 )  ( 2i+3 )  ( 2i+1 ) }-  (\frac{3}{280}\delta_{m,0}+\frac{1}{120}\delta_{m,1})\delta_{i,3},  & j=i+3,\\[0.3em]
-\frac { ( i-3 ) i ( i+1 )  ( i-m )  ( i+m ) }{32 ( 2i-3 )  ( 2i-1 )  ( 2i+3)  ( 2i+1 ) }-  (\frac{1}{280}\delta_{m,0} +\frac{1}{315}\delta_{m,1}) \delta_{i,3},  & j=i+4,
\\[0.3em]
0, & j\ge i+5.
 \end{cases}
 \end{align}
 \end{lemma}
We postpone the proof to Appendix \ref{App:B}.
 
We shall look for
  \begin{align}
\psi_{mN}=\sum_{i=N_m}^{N}\hat u_{i}^m\phi_i^m.\label{express1}
 \end{align}
 Now, plugging the expression of \eqref{express1} in \eqref{weak3}, and taking $v_{N}$
through all the basis functions in $X_N^m$, we will arrive at the following algebraic linear eigenvalue
system:
 \begin{align}
A^m \hat u^m=\lambda_{N}^m B^m \hat u^m,
\end{align}
with
 \begin{align}
&A^m=(A_{ij}^m)_{N_m\le i,j\le N}, \qquad B^m=(B_{ij}^m)_{N_m\le i,j\le N},\qquad 
\hat u^m=(\hat  u_{N_m}^m,\cdots,\hat  u_{N}^m)^{\tr},
\end{align}
which  can be efficiently solved.

\section{Extension to ellipitc domain}
In the section, we extend our algorithm  and numerical analysis from a circular disk to an elliptic domain,
$$\Omega=\Big\{ (x,y): \frac{x^2}{a^2}+ \frac{y^2}{b^2}<1  \Big\},$$ 
where $a$ and $b$  are  the semi-major axis and the semi-minor axis, respectively,
i.e., $a\ge b> 0$.

\subsection{Pole conditions }
Let us make the polar transformation  $(x,y)=(ar\cos\theta,br\sin\theta)$,
which  maps
the rectangle $R=\{(r,\theta):  0\le r<1,  0\le \theta<2\pi\}$ in polar coordinates
onto the ellipse $\Omega$ in  Cartesian coordinates.
For $s\ge 0$, we denote  $\OH^s(R)=\{\tilde u(r,\theta)=u(a r\cos\theta,br\sin \theta):  u \in H^s_0(\Omega)  \}$,
 which is equipped with the norm $ \| u\|_{s}$. 
If no confusion would arise, we shall also use
the notation $u$ for  its correspondence $\tilde u$ on $R$.

%To reveal the structure of $\OH^2(D)$, we first revisit 
We now revisit the gradient and Laplacian in Cartesian coordinates.
It is readily checked that 
\begin{align}
\nabla u=&\big(\frac{1}{a}\big(\cos\theta\partial_r u-\frac{1}{r}\sin\theta\partial_\theta u),\frac{1}{b}(\sin\theta\partial_ru+\frac{1}{r}\cos\theta\partial_\theta u)\big)^{\tr}, \label{a5.1} \\
\begin{split}
\Delta u%=&\Big(\frac{\cos^2\theta}{a^2}+\frac{\sin^2\theta}{b^2}\Big)\partial_r^2u+\Big(\frac{1}{a^2}-\frac{1}{b^2}\Big) \frac{\sin 2\theta}{r}\Big(\frac{1}{r}\partial_\theta u-\partial_r\partial_\theta u\Big)
%+\Big(\frac{\sin^2\theta}{a^2}+\frac{\cos^2\theta}{b^2}\Big) \frac{1}{r}\Big(\partial_r u+\frac{1}{r}\partial^2_\theta u\Big) \\
=&\frac12\Big(\frac{1}{a^2}+\frac{1}{b^2}\Big)\Big( \partial_r^2u + \frac{1}{r}\partial_r u+\frac{1}{r^2}\partial^2_\theta u\Big)
\\
&+ \frac12\Big(\frac{1}{a^2}-\frac{1}{b^2}\Big)\Big[ \cos2\theta \Big( \partial_r^2u - \frac{1}{r}\partial_r u-\frac{1}{r^2}\partial^2_\theta u\Big)
+\frac{2\sin 2\theta}{r}\Big(\frac{1}{r}\partial_\theta u-\partial_r\partial_\theta u\Big)\Big]
.
\end{split}\label{a5.2}
\end{align}
Specifically,
\iffalse for any  function
 $$
u  = \sum_{m=-\infty}^{\infty} u_m(r) \e^{\i m \theta},
$$
it holds that
\begin{align}
\label{Nablaue}
\begin{split}
 \nabla u&=\Big( \frac{1}{2a} \sum_{m=-\infty}^{\infty}  \big( u_m' -\frac{m}{r} u_m  \big)
 \e^{\i (m+1) \theta}+ \frac{1}{2a}\sum_{m=-\infty}^{\infty}  \big( u_m' +\frac{m}{r} u_m  \big)
 \e^{\i (m-1) \theta},  
 \\
 &\quad   -\frac{\i}{2b} \sum_{m=-\infty}^{\infty} \big( u_m' -\frac{m}{r} u_m  \big)
 \e^{\i (m+1) \theta} +\frac{\i}{2b} \sum_{m=-\infty}^{\infty} ( u_m' +\frac{m}{r} u_m  \big)
 \e^{\i (m-1) \theta}\Big)^{\tr},
 \end{split}
\\
\label{Deltaue}
\begin{split}
  \Delta u& = \frac12\Big( \frac{1}{a^2}+\frac{1}{b^2}\Big) \sum_{m=-\infty}^{\infty}  \Big( \partial_r^2 + \frac{1}{r}\partial_r -\frac{m^2}{r^2}\Big) u_m(r) \e^{\i m\theta}
  \\
&+\frac14\Big(\frac{1}{a^2}-\frac{1}{b^2}\Big)
\sum_{m=-\infty}^{\infty} \Big( \partial_r^2-\frac{1+2m}{r}\partial_r +\frac{m^2+2m}{r^2}
%-\frac{2m}{r}\partial_r+\frac{2m}{r^2} 
\Big)u_m(r) \e^{\i (m+2)\theta}
  \\
&+\frac14\Big(\frac{1}{a^2}-\frac{1}{b^2}\Big) \sum_{m=-\infty}^{\infty} 
\Big( \partial_r^2-\frac{1-2m}{r}\partial_r +\frac{m^2-2m}{r^2}
%+\frac{2m}{r}\partial_r-\frac{2m}{r^2} 
\Big)u_m(r) \e^{\i (m-2)\theta}.
\end{split}
\end{align}
\fi
\begin{align}
\label{Nablaue}
\begin{split}
 \nabla [u_m(r) \e^{\i m\theta}]&=\Big( \frac{1}{2a}  \big( u_m' -\frac{m}{r} u_m  \big)
 \e^{\i (m+1) \theta}+ \frac{1}{2a}   \big( u_m' +\frac{m}{r} u_m  \big)
 \e^{\i (m-1) \theta},  
 \\
 &   -\frac{\i}{2b} \big( u_m' -\frac{m}{r} u_m  \big)
 \e^{\i (m+1) \theta} +\frac{\i}{2b} ( u_m' +\frac{m}{r} u_m  \big)
 \e^{\i (m-1) \theta}\Big)^{\tr},
 \end{split}
\\
\label{Deltaue}
\begin{split}
  \Delta [u_m(r) \e^{\i m\theta}] &= \frac12\Big( \frac{1}{a^2}+\frac{1}{b^2}\Big) \mathcal{L}_m u_m(r) \e^{\i m\theta}
  \\
&+\frac14\Big(\frac{1}{a^2}-\frac{1}{b^2}\Big) \Big[
 \mathcal{K}_m u_m(r) \e^{\i (m+2)\theta} +  \mathcal{K}_{-m} u_m(r) \e^{\i (m-2)\theta}
\Big],
\end{split}
\end{align}
where $\mathcal{L}_m$ and $ \mathcal{K}_m$ are differential operators defined by
\begin{align*}
\mathcal{L}_m= \partial_r^2 + \frac{1}{r}\Big(\partial_r -\frac{m^2}{r}\Big),
\qquad \mathcal{K}_m= \partial_r^2 - \frac{1+2m}{r}\partial_r +\frac{m^2+2m}{r^2}
=\mathcal{L}_m-\frac{2m}{r} \Big(\partial_r -\frac{1}{r}\Big).
\end{align*}

%To ensure the well-posedness of the bilinear form $a(\cdot,\cdot)$ and $b(\cdot,\cdot)$, 
To make the $\nabla  [u_m(r) \e^{\i m\theta}] $ and $\Delta  [u_m(r) \e^{\i m\theta}] $ meaningful at the origin, 
% the well-posedness of the bilinear form $a(\cdot,\cdot)$ and $b(\cdot,\cdot)$, 
one requires that
\begin{align}
\label{cond2a}
mu_m(0)=0,\qquad \lim_{r\to 0+}\Big(u_m'(r)-\frac{m^2u_m(r)}{r}\Big)  = 0,
%, \qquad (1\pm 2m) \partial_ru_m - \frac{m\pm 2}{r}u_m
,\qquad m  \lim_{r\to 0+}  \Big(u_m'(r) - \frac{u_m(r)}{r}\Big) = 0
%= \big(\partial_r u_m-\frac{m^2}{r}u_m\big) \pm 2m \big(\partial_r u_m-\frac{1}{r}u_m\big) 
\end{align}
which, as before,  can be further simplified into the following three categories,
\begin{align}
&(1). \quad {u}_m'(0)=0,  && m=0;\label{cond3a}\\
&(2). \quad {u}_m(0)=0,   && |m|=1;\label{cond4a}\\
&(3). \quad {u}_m(0)={u}_m'(0)=0, && |m|\geq2.\label{cond5a}
\end{align}

\subsection{Specral-Galerkin approximation and implementation}
Define the approximation spaces,
\begin{align*}
 X_{N} =X_{N/2,N},\qquad 
X_{M,N} = \mathrm{span}\big\{ u_m(r) \e^{\i m \theta}:   u_m\in X_N^m,  -M \le m \le M    \big\}.
\end{align*}
Then the  spectral-Galerkin approximation to \eqref{weak0} reads: Find $(\lambda_N,\psi_{N})\in \RR\times X_{N}$ such that $\|\nabla u\|=1$ and 
\begin{align}
\label{Galellipse}
(\Delta \psi_{N},\Delta v)=\lambda_{N} (\nabla \psi_{N},\nabla v), \qquad 
v\in X_{N}.
\end{align}

We now give a brief explanation on how to solve the problems \eqref{Galellipse} efficiently.   
Define the matrices  $A^{m,n}, B^{m,n}\in \RR^{N-N_m+1,N-N_n+1}$  with their 
entries 
\begin{align*}
& A^{m,n}_{j,k} = 0, \qquad |m-n|\not\in \{ 0,2,4\},\\
&A^{m,m}_{j,k}  = A^{m,m}_{k,j} = \frac{\pi}{2}\big(\frac{1}{a^2}+\frac{1}{b^2}\Big)^2 \big( \mathcal{L}_m \phi^m_k,  \mathcal{L}_m \phi^m_j  \big)_{\omega^{0,1},I}  
\\
&\qquad + \frac{\pi}{8} \Big(\frac{1}{a^2}-\frac{1}{b^2}\Big)^2 \Big[ \big( \mathcal{K}_m \phi^m_k,   \mathcal{K}_m\phi^m_j  \big )_{\omega^{0,1},I}  
+ \big( \mathcal{K}_{-m} \phi^m_k,   \mathcal{K}_{-m}\phi^m_j  \big )_{\omega^{0,1},I}\Big]
,
\\
&A^{m,m+2}_{j,k}  = A^{m+2,m}_{k,j} 
 =\frac{\pi}{4}\Big(\frac{1}{a^4}-\frac{1}{b^4}\Big) \Big[   \big( \mathcal{L}_{m+2}\phi^{m+2}_k,  \mathcal{K}_m \phi^m_j  \big )_{\omega^{0,1},I}  
+  \big(\mathcal{K}_{-m-2}\phi^{m+2}_k,   \mathcal{L}_m \phi^m_j  \big)_{\omega^{0,1},I} \Big] ,
\\
&A^{m,m+4}_{j,k}  =  A^{m+4,m}_{k,j}  =
\frac{\pi}{8}\Big(\frac{1}{a^2}-\frac{1}{b^2}\Big)^2 \big(\mathcal{K}_{-m-4}\phi^{m+4}_k,  \mathcal{K}_m \phi^m_j   \Big)_{\omega^{0,1},I},  
\end{align*}
and 
\begin{align*}
& B^{m,n}_{j,k} = 0, \qquad |m-n|\not\in \{ 0,2\},\\
&B^{m,m}_{j,k} =B^{m,m}_{k,j}  =\pi \Big(\frac{1}{a^2}+\frac{1}{b^2}\Big) \Big[ \big( \partial_r\phi^m_k,   \partial_r\phi^m_j  \big)_{\omega^{0,1},I}  
+  m^2 \big(\phi^m_k,   \phi^m_j  \big)_{\omega^{0,-1},I}  \Big]
,
\\
&B^{m,m+2}_{j,k} =  B^{m+2,m}_{k,j}  =  \frac{\pi}{2} \Big(\frac{1}{a^2}+\frac{1}{b^2}\Big) \Big( \big( \partial_r\phi^{m+2}_k +\frac{m+2}{r}\big)\phi^{m+2}_k,   \big(\partial_r-\frac{m}{r}\big) \phi^m_j  \Big)_{\omega^{0,1},I} .
\end{align*}
In view of \eqref{phid2}-\eqref{phi31}, all the nontrivial matrices are penta-diagonal,
and their nonzero entries can all be evaluated analytically.
Further suppose
$$
\psi_N = \sum_{m=-N/2}^{N/2} \sum_{k=N_m}^{N} \hat u_{k}^m \phi^m_k(r) \e^{\i m \theta}.
$$
We  arrive at the following algebraic eigenvalue
problem:
 \begin{align}
\bs A \bs {\hat u}=\lambda_{N}\bs B \bs {\hat u},
\end{align}
where  $\bs {\hat u}$ is  the unknown vector
$$
\bs{\hat u} =  ((\hat  u^{-N/2})^{\tr}, (\hat  u^{1-N/2})^{\tr}, \cdots, (\hat  u^{N/2} )^{\tr} )^{\tr},
 \qquad \hat u^m = ( \hat u^m_{N_m}, \hat u^m_{N_m+1}, \cdots, \hat u^m_{N} )^{\tr},
$$
and $\bs A$ and $\bs B$ are  block hepta-diagonal and block penta-diagonal matrices, respectively,
\begin{align*}
\bs A = \big[ A^{m,n} \big]_{-N/2\le m,n\le N/2}\quad \text{ and }\quad  
\bs B =   \big[ B^{m,n} \big]_{-N/2\le m,n\le N/2}.
\end{align*}

\subsection{Error estimates}

We now conduct the error analysis of \eqref{weak0} by using the standard  theory of Babu\v{s}ka and J. Osborn \cite{babuska1991}.
To this end,
we  first
define the semi-norm $ |\cdot |_{s,*} $ in $H^{s}(\Omega)$ with $s\ge 2$,
\begin{align*}
 |u|_{s,*,\Omega} =  \Big[  \sum_{\nu=0}^s \big\|\partial_x^{\nu}\partial_y^{s-\nu}u\big\|_{\varpi^{s-2}}^2  + \sum_{\nu=0}^2  \big\| (a^2y\partial_x- b^2x\partial_y )^{s-2}  \partial_x^{\nu}\partial_y^{2-\nu} u \big\|^2 \Big]^{\frac12},
\end{align*}
   where  the weight function $\varpi^{\alpha}=\varpi^{\alpha}(x,y) = \varpi^{\alpha}(r):= (1-r^2)^{\alpha}$.

\begin{theorem}
Suppose $u\in H_0^2(\Omega)\cap H^{s}(\Omega)$ for $s\ge 2$. Then it holds that
   \begin{align}
   \label{Inferr}
   \begin{split}
   \inf_{v\in X_N} \|\Delta (u-v)\| \lesssim&\, N^{2-s}  |u|_{s,*}.
%   \\
 %   \lesssim&\, N^{2-s} \|u\|_{s}
    \end{split}
   \end{align}
\end{theorem}
\begin{proof}
We first note that
\begin{align*}
 \|\Delta u\|^2 =&\,\int_{\Omega} \Big[\big|\partial_x^2u\big|^2+ (\partial_x^2u\partial_y^2\overline{u}+ \partial_y^2u\partial_x^2\overline{u} ) + \big|\partial_y^2u\big|^2
 \Big] \d x \d y
 \\  =&\, \|\partial_x^2u\|^2 + 2\|\partial_x\partial_y u \|^2 + \|\partial_y^2 u\|^2, \qquad u\in H_0^2(\Omega),
\end{align*}
where we derive the second equality sign by integration by parts.
Owing to the linear mapping  $(x,y)\mapsto (ax,by)$  from $D$ onto $\Omega$,  it suffices to prove \eqref{Inferr} for $\Omega$ being the unit disk $D$, i.e.,
\begin{align}
\label{InferrX}
   \begin{split}
   \inf_{v\in X_N} |  u-v|_{2,D} \lesssim&\, N^{2-s} \Big[  \sum_{\nu=0}^s \big\|\partial_x^{\nu}\partial_y^{s-\nu} u\big\|_{\varpi^{s-2},D}  + \sum_{\nu=0}^2  \big\| (y\partial_x- x\partial_y )^{s-2}  \partial_x^{\nu}\partial_y^{2-\nu} u \big\|_{D} \Big].
    \end{split}
    \end{align}
To this end,  we further denote 
  $\PP_N^{-2}(D)= \PP_N(D)\cap H_0^2(D)$ and find  that
$$ (1-x^2-y^2)^2  (x+ \i y )^{m} (x- \i y )^n =(1-r^2)^2 r^{m+n}  \e^{\i (m-n)\theta},
\quad m,n \in \NN_0,\   (x,y)\in D.$$
It is then obvious that $\PP_{\min(N,N/2+4)}^{-2}(D)\subset X_{N/2,N}=X_N$
and 
$$
 \inf_{v\in X_N} |  u-v|_{2,D}  \le  \inf_{v\in \PP_{N_*}^{-2}(D)} |  u-v|_{2,D} , \qquad N_*=\min(N,N/2+4).
$$
Thus we deduce \eqref{InferrX}  from the following error estimate on  polynomial approximations 
\cite[Theorem 4.3]{LX15},
\begin{align*}
 \inf_{v\in \PP_N^{-2}} |  u-v|_{2,D} \lesssim\, N^{2-s} \Big[  \sum_{\nu=0}^s \big\|\partial_x^{\nu}\partial_y^{s-\nu}u\big\|_{\varpi^{s-2},D}  + \sum_{\nu=0}^2  \big\| (y\partial_x- x\partial_y )^{s-2}  \partial_x^{\nu}\partial_y^{2-\nu} u \big\|_{D} \Big].
\end{align*}
The proof is now completed.
\end{proof}

By the approximation theory of  Babu\v{s}ka and Osborn on 
the Ritz method for self-adjoint and positive-definite  eigenvalue problems 
\cite[pp. 697-700]{babuska1991},  
we now arrive at the following main theorem.
\begin{theorem}
Let $\left\{ \lambda_{i,N}\right\}$ be 
the eigenvalues of \eqref{Galellipse}  ordered non-decreasingly with respect to $i$, repeated according to their multiplicities.
Further let $\lambda_k$ be an eigenvalue of \eqref{weak0} with the geometric multiplicity $q$  and assume that
$\lambda_k=\lambda_{k+1}=\dots=\lambda_{k+q-1}$.
   Then there exits a constant $C>0$ such that
   \begin{align*}
           &\lambda_k \le \lambda_{j,N} \le \lambda_k + C N^{4-2s} \sup_{\psi\in E(\lambda_k)}  | \psi|_{s,*,\Omega}^2, \quad  j=k,k+1,\dots,k+q-1,
   \end{align*}
   where
   \begin{align*}
  E(\lambda_k) := \left\{ \psi \text{ is an eigenfunction corresponding to } \lambda_k    \text{ with } \|\psi\|_{1,\Omega}=1     \right\} .
\end{align*}

    Let $\psi_{j,N}$ be an eigenfuction corresponding to $\lambda_{j,N}$ for $j=k,k+1,\dots,k+q-1$, then there exists a constant $C$ such that
  \begin{align*}
           \inf_{u\in E(\lambda_k)} \|u-\psi_{j,N}\|_{2,\Omega} \le  C N^{2-s} \sup_{\psi\in E(\lambda_k)}  | \psi|_{s,*,\Omega}.
   \end{align*}
    Let $\psi_{k}$ be an eigenfuction corresponding to $\lambda_{k}$,  then there exist  a constant $C$
   and a function  $v_N\in \mathrm{span}\{ \psi_{k,N},\dots,\psi_{k+q-1,N} \} $ such that
  \begin{align*}
         \|\psi_k-v_{N}\|_{2,\Omega} \le  C N^{2-s} \sup_{\psi\in E(\lambda_k)}   | \psi|_{s,*,\Omega}.
   \end{align*}

\end{theorem}

%\subsection{Efficient implementation of the algorithm}

\section{Numerical experiments}
\indent We now perform a sequence of numerical tests to study the convergence behavior
and show the effectiveness of our algorithm. We operate our programs in MATLAB 2015b.
\subsection{Circular disk}
\subsubsection{Spectral analysis}
We now turn to the spectral decomposition of \eqref{eqn:eigensys4}--\eqref{eqn:eigensys5}.
Under the polar coordinates, we first  reformulate \eqref{eqn:eigensys4} as follows,
\begin{align}
\label{eq:4thpolar}
\big[ r^4 \partial_r^2 + 2r^3\partial_r^3 + (\lambda r^2+ 2\partial_{\theta}^2-1) r^2\partial_r^2+ (\lambda r^2   - 2\partial_{\theta}^2+1)  r\partial_r
 + ( \lambda r^2 +4+\partial_{\theta}^2) \partial_{\theta}^2\big]\psi(r,\theta) = 0.
\end{align}
%Owing to the periodicity of $\psi$  in $\theta$,  
We next expand  $\psi$  in the Fourier series in $\theta$,
\begin{align*}
  \psi(r,\theta) =  \sum_{m=-\infty}^{\infty} \psi_m(r) e^{im\theta}.
\end{align*}
Then \eqref{eq:4thpolar} is reduced to 
\begin{align}
\label{eq:4thpolar2}
\begin{split}
\big[ r^4 \partial_r^2 + 2r^3\partial_r^3 + (\lambda r^2- 2m^2-1) r^2\partial_r^2+ (\lambda r^2   + 2m^2+1)  r\partial_r
\\
 - ( \lambda r^2 +4-m^2)m^2\big]\psi_m(r) = 0, \quad \forall m \in \mathbb{Z},
 \end{split}
\end{align}
\iffalse
Making the  variable transformation $\rho=\sqrt{\lambda}r$ and setting $\phi_m(\rho) =\psi_m(r)$, we further simplify
\eqref{eq:4thpolar2}  as 
\begin{align}
\label{eq:4thpolar3}
\begin{split}
\big[ \rho^4& \partial_{\rho}^2 + 2{\rho}^3\partial_{\rho}^3 + ({\rho}^2- 2m^2-1) {\rho}^2\partial_{\rho}^2+ ( {\rho}^2   + 2m^2+1)  {\rho}\partial_{\rho}
 - ({\rho}^2 +4-m^2)m^2\big]\phi_m(\rho)
 \\
 =&  \big[ \rho^2\partial_{\rho}^2 -3\rho \partial_{\rho} + (4-m^2)  \big]  \big[ \rho^2\partial_{\rho}^2 +\rho \partial_{\rho} + (\rho^2-m^2)   \big] \phi_m(\rho)
 \\
 =&  \big[ \rho^2\partial_{\rho}^2 -3\rho \partial_{\rho} + (\rho^2+4-m^2)  \big] \big[ \rho^2\partial_{\rho}^2 +\rho \partial_{\rho} -m^2   \big] \phi_m(\rho)
 = 0.
 \end{split}
\end{align}
which, together with the pole conditions \eqref{cond3}-\eqref{cond5}, implies $\psi_m$  is a combination of the
monomial of degree $m$  and the $m$-th order  Bessel function of the first kind,
\begin{align*}
       \psi_m(r) = \phi(\rho) = c_{m,1}  \rho^m + c_{m,2} J_m(\rho) = c_{m,1}  (\sqrt{\lambda}r)^m + c_{m,2} J_m(\sqrt{\lambda}r).
\end{align*}
\fi
which, together with the pole conditions \eqref{cond3}-\eqref{cond5}, admits a general solution,
\begin{align*}
       \psi_m(r)  = c_{m,1}  (\sqrt{\lambda}r)^m + c_{m,2} J_m(\sqrt{\lambda}r).
\end{align*}

%We now determine the eigenvalue $\lambda$ and the nontrivial coefficients $c$'s by the boundary conditions \eqref{eqn:eigensys5}. In polar coordinates, 
Meanwhile, the boundary conditions  $ \psi_m(1)=\psi_m'(1)=0$  imply that
\begin{align}
\label{eq:system}
\begin{cases}
c_{m,1}   \sqrt{\lambda}^m  + c_{m,2} J_m(\sqrt{\lambda}) =0 ,\\
c_{m,1}  m \sqrt{\lambda}^m + c_{m,2} \sqrt{\lambda}J_m'(\sqrt{\lambda}) =0,
\end{cases}
\end{align}
with some nontrivial $c'$s. As a results, the determinant 
\begin{align}
\label{lambda}
 \lambda^{\frac{m}{2}} \big[  \sqrt{\lambda}J_m'(\sqrt{\lambda}) - m J_m(\sqrt{\lambda})\big]
% =  \frac{\lambda^{\frac{m}{2}} }2\big[  \sqrt{\lambda}J_{m-1}(\sqrt{\lambda}) -  \sqrt{\lambda}J_{m+1}(\sqrt{\lambda}) - 2m J_m(\sqrt{\lambda}) \big] \\
=- \lambda^{\frac{m+1}{2}} J_{m+1}(\sqrt{\lambda})=0,
\end{align}
where the second equality sign is derived from the recurrence relation (4) in Page 45 of \cite{Watson}.
 In return, 
 the fundamental solution of \eqref{eq:system}
%we get a fundamental solution of \eqref{eq:system},
%\begin{align*}
%    (c_{m,1},c_{m,2}) = (J_m(\sqrt{\lambda_{m,k}}), - \sqrt{\lambda_{m,k}}^m  ). 
%\end{align*}
%Finally, we derive the eigen-solution of  \eqref{eqn:eigensys4}--\eqref{eqn:eigensys5},
determines the corresponding eigenfunction of  \eqref{eqn:eigensys4},
\begin{align*}
\psi_{m}(r,\theta) = \big[J_m(\sqrt{\lambda_{m}}) (\sqrt{\lambda_{m}}\,r)^m -  \sqrt{\lambda_{m}}^m 
J_m(\sqrt{\lambda_{m}}\,r)\big] \e^{\i m\theta}.
\end{align*}

Finally, we note that the nontrivial roots of \eqref{lambda} define a  sequence of increasingly ordered eigenvalues $0<\lambda_{m}^1<\lambda_{m}^2 <\lambda_{m}^3<\dots$, which are exactly the eigenvalues of the following second-order equation 
stemmed from the Laplacian eigenvalue problem on the unit disk,
$$
\Big(\partial_r^2 +\frac{1}{r}\partial_r -\frac{(1+|m|)^2}{r^2} \Big) \phi_m = \mu_m\phi_m.
$$

\subsubsection{Numerical results}

We take  $m=0,1,2$ as our examples.   The numerical results of first four eigenvalues for different $m$ and $N$ are listed in Table 5.1-5.4.
\begin{center}{\small\bf Table 5.1 The first four eigenvalues for $m=0$ and different $N$
in the unit disk. }
\begin{tabular}{{ccccc}}
  \hline
\multicolumn{1}{c} N&$\lambda_{0N}^1$&      $\lambda_{0N}^2$&    $\lambda_{0N}^3$&   $\lambda_{0N}^4$\\ \hline
\multicolumn{1}{c}{10}&    14.6819706421365&  49.2184567483993&       103.5024835613828&  177.6009453441972\\
\multicolumn{1}{c}{20}&   14.6819706421239&   49.2184563216945&    103.4994538951366& 177.5207668138042\\
\multicolumn{1}{c}{30} &  14.6819706421239&  49.2184563216945&     103.4994538951365&  177.5207668138044\\
  \hline
\end{tabular}\end{center}
\begin{center}{\small\bf Table 5.2 The first four eigenvalues for $m=1$ and different $N$ in 
the unit disk. }
\begin{tabular}{{ccccc}}
  \hline
\multicolumn{1}{c} N&$\lambda_{1N}^1$&      $\lambda_{1N}^2$&    $\lambda_{1N}^3$&   $\lambda_{1N}^4$\\ \hline
\multicolumn{1}{c}{20}&   26.3746164271634&   70.8499989190960&    135.0207088659703& 218.9201891456649\\
\multicolumn{1}{c}{30} & 26.3746164271634&  70.8499989190958&     135.0207088659704&  218.9201891456624\\
\multicolumn{1}{c}{40} & 26.3746164271634&  70.8499989190957&     135.0207088659696&  218.9201891456631\\
\multicolumn{1}{c}{50} & 26.3746164271634&  70.8499989190957&     135.0207088659700&  218.9201891456630\\
  \hline
\end{tabular}\end{center}
\begin{center}{\small\bf Table 5.3 The first four eigenvalues for $m=2$ and different $N$
in the unit disk.}
\begin{tabular}{{ccccc}}
  \hline
\multicolumn{1}{c} N&$\lambda_{2N}^1$&      $\lambda_{2N}^2$&    $\lambda_{2N}^3$&   $\lambda_{2N}^4$\\ \hline
\multicolumn{1}{c}{30} &  40.7064658182003&  95.2775725440372&    169.3954498260997&  263.2008542550081\\
\multicolumn{1}{c}{40} & 40.7064658182002&  95.2775725440371&    169.3954498260988&  263.2008542550071\\
\multicolumn{1}{c}{50} & 40.7064658182004&  95.2775725440372&    169.3954498260995&  263.2008542550078\\
\multicolumn{1}{c}{60} & 40.7064658182003&  95.2775725440370&    169.3954498260994&  263.2008542550076\\
  \hline
\end{tabular}\end{center}

We know from { Tables 5.1--5.3 that numerical} eigenvalues achieve at least fourteen-digit accuracy with $N\geq 20$ { for $m=0$ and $N\geq 40$ for $m=1,2$, respectively}. If we choose the { numerical solutions with} $N = 60$ as reference solutions, the error figures of the approximate eigenvalue {$\lambda_{mN}^i(m=0,1,2; i=1,2,3,4)$}  with different $N$ are listed in {Figures 1-3}.

 It is worthy to note that, when imposing the pole conditions $\psi_m'(0)=\psi_m(0)=0$ 
as in \cite{shen1997} for $m=1$,
one would necessarily get spurious eigenvalues even for large $N,$ which can only serve as  upper bounds of each the exact ones.  
For instance, the first computational eigenvalue in this case reads $28.7378$, a number  far away from the reference one $26.3746$.

\begin{figure}[h!]
\begin{minipage}[c]{0.48\textwidth}
\centering
\includegraphics[width=1\textwidth]{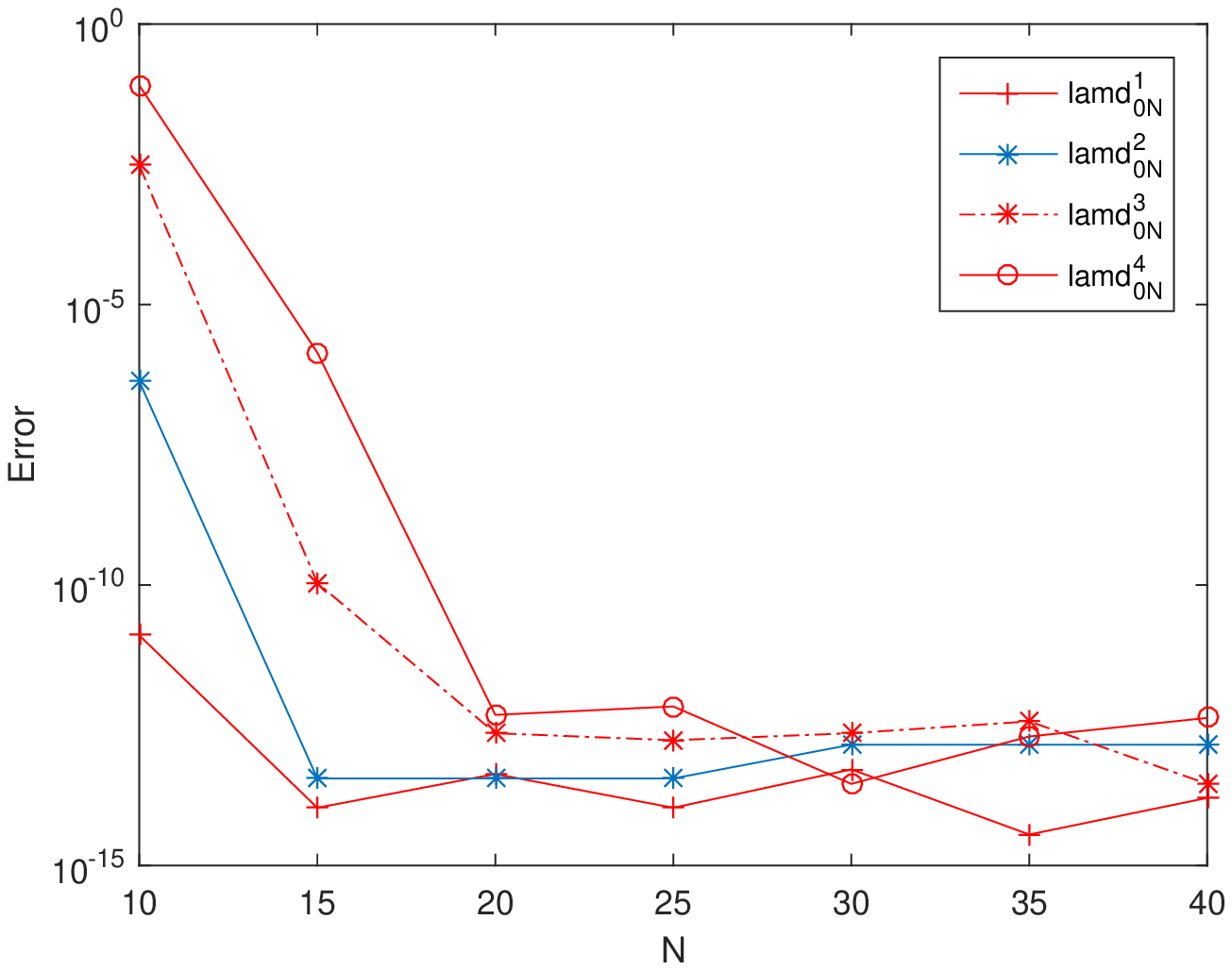}
\caption{Errors between numerical solutions and the reference solution for $m=0$.}\label{fig1}
\end{minipage}\hfill%
\begin{minipage}[c]{0.48\textwidth}
\centering
\includegraphics[width=1\textwidth]{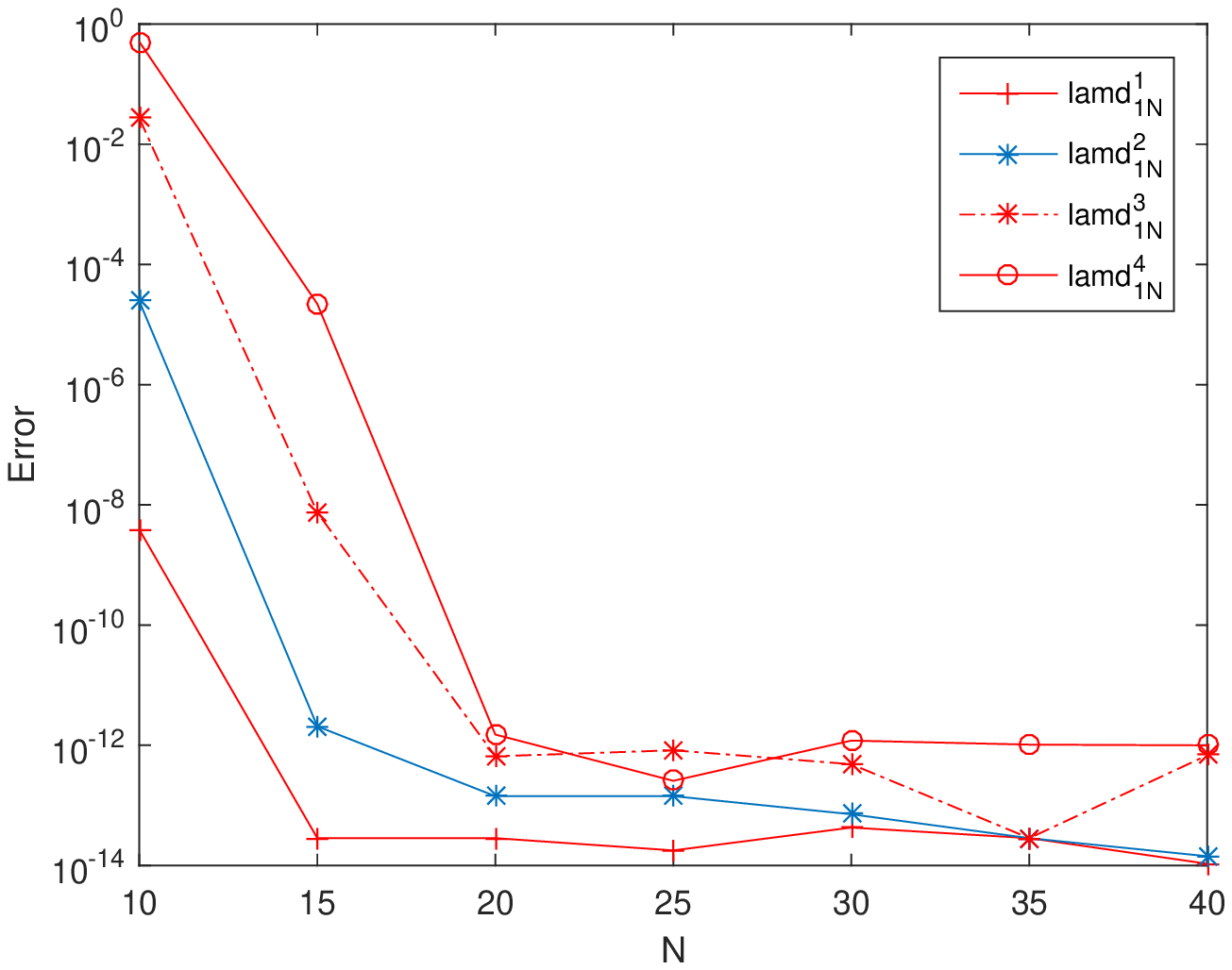}
\caption{Errors between numerical solutions and the reference solution for $m=1$.}\label{fig2}
\end{minipage}\hfill%
\end{figure}

\begin{figure}[h!]
\begin{minipage}[c]{0.48\textwidth}
\centering
\includegraphics[width=1\textwidth]{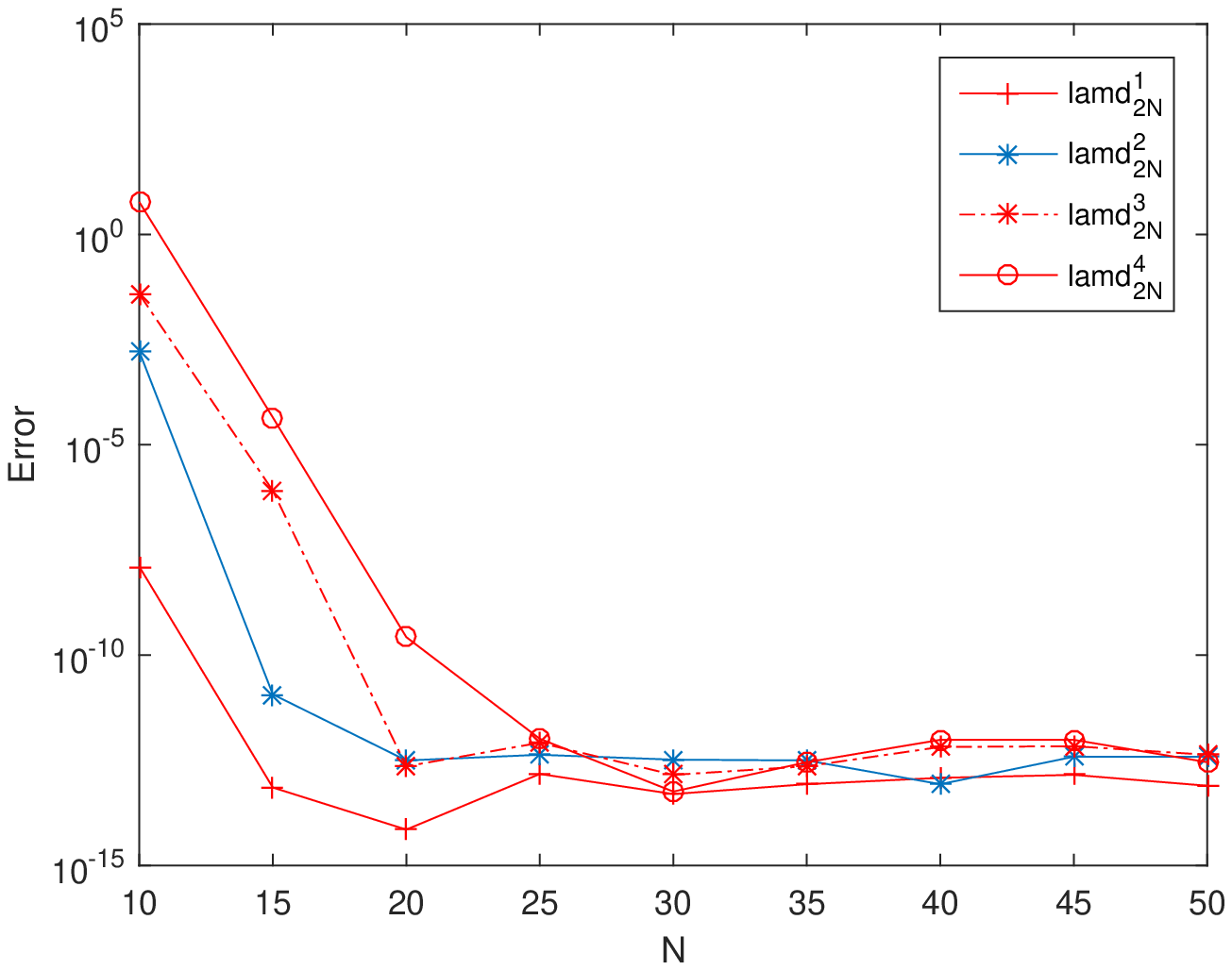}
\caption{Errors between numerical solutions and the reference solution for $m=2$ on the unit disk.}\label{fig3}
\end{minipage}\hfill%
\begin{minipage}[c]{0.48\textwidth}
\centering
\includegraphics[width=1\textwidth]{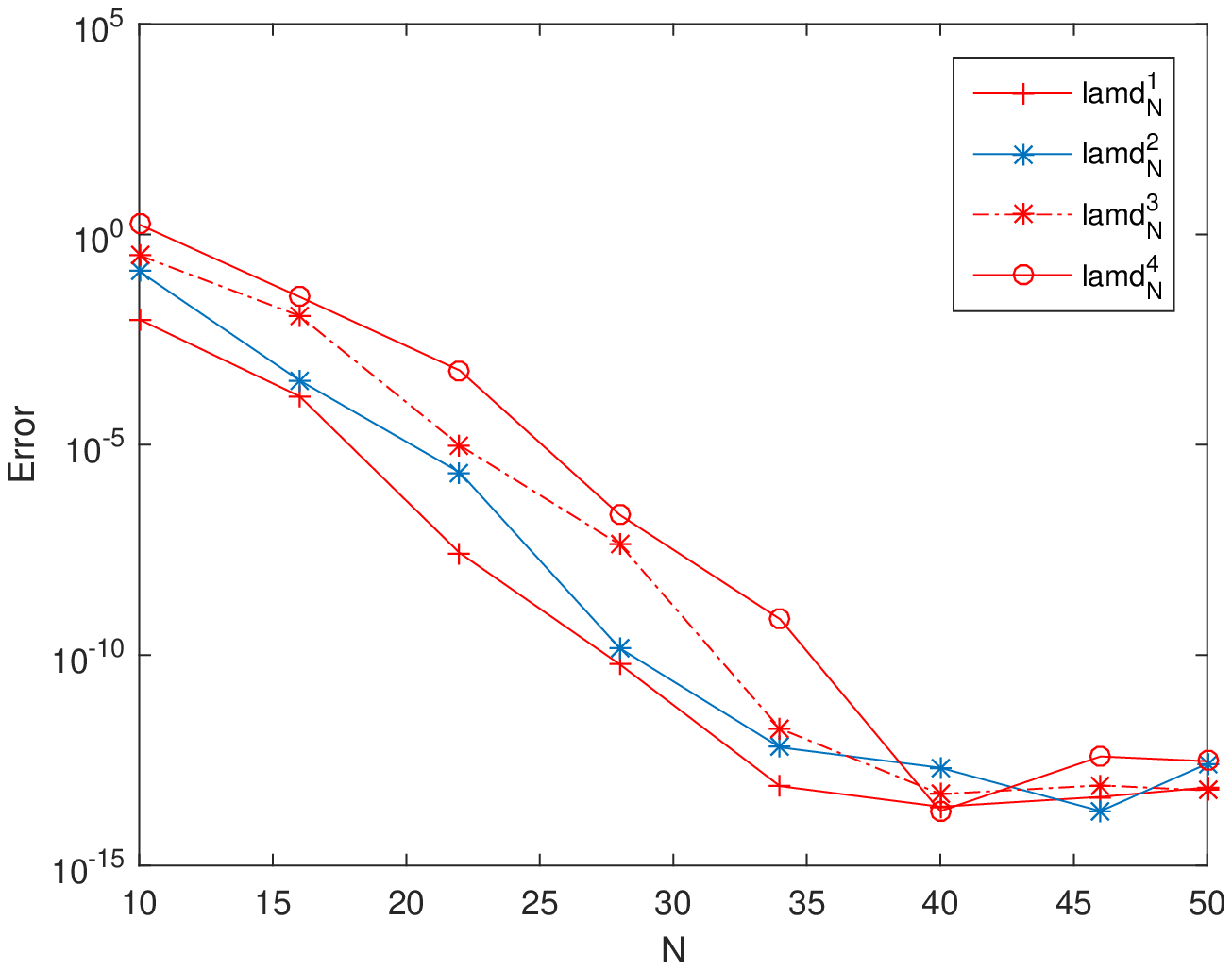}
\caption{Errors between numerical solutions and the reference solution on the elliptic domain.}\label{fig4}
\end{minipage}\hfill%
\end{figure}

\subsection{Elliptic domain }
We take $a=3, b=1$ as our example. The numerical {data of the } first four eigenvalues { are listed in Table 5.4. We see} that the eigenvalues achieve at least fourteen-digit accuracy with $N\geq 40$. If we choose the solutions of $N = 60$ as reference solutions, the error figures of the approximate eigenvalue $\lambda_{N}^i(i=1,2,3,4) $  with different $N$ are listed in { Figure 4}.
\begin{center}{\small\bf Table 5.4 The first four eigenvalues for different $N$
in an elliptic domain with  $a=3, b=1$. }
\begin{tabular}{{ccccc}}
  \hline
\multicolumn{1}{c} N&$\lambda_{N}^1$&      $\lambda_{N}^2$&    $\lambda_{N}^3$&   $\lambda_{N}^4$\\ \hline
\multicolumn{1}{c}{20} &  9.96633619654313 &  11.0706597920227&    13.1630821009849&    15.6448857440637 \\
\multicolumn{1}{c}{30} &  9.9663343484475 &   11.0706554383893&    13.1627539459867&    15.6437495616386 \\
\multicolumn{1}{c}{40} &  9.96633434844728 &  11.0706554383168&   13.1627539455290&     15.6437494538630 \\
\multicolumn{1}{c}{50} &  9.96633434844729 &  11.0706554383167&     13.1627539455291&    15.6437494538630\\
\multicolumn{1}{c}{60} &  9.96633434844726 &  11.0706554383166&   13.1627539455290&    15.6437494538630 \\
  \hline
\end{tabular}\end{center}

 Before concluding this section, we would like to present some figures of the (real) eigenfunctions corresponding to the
smallest 8 eigenvalues in Figures \ref{fig:psi1}-\ref{fig:psi8}.

\begin{figure}[h!]
\begin{minipage}[c]{0.48\textwidth}
\centering
\includegraphics[width=1\textwidth]{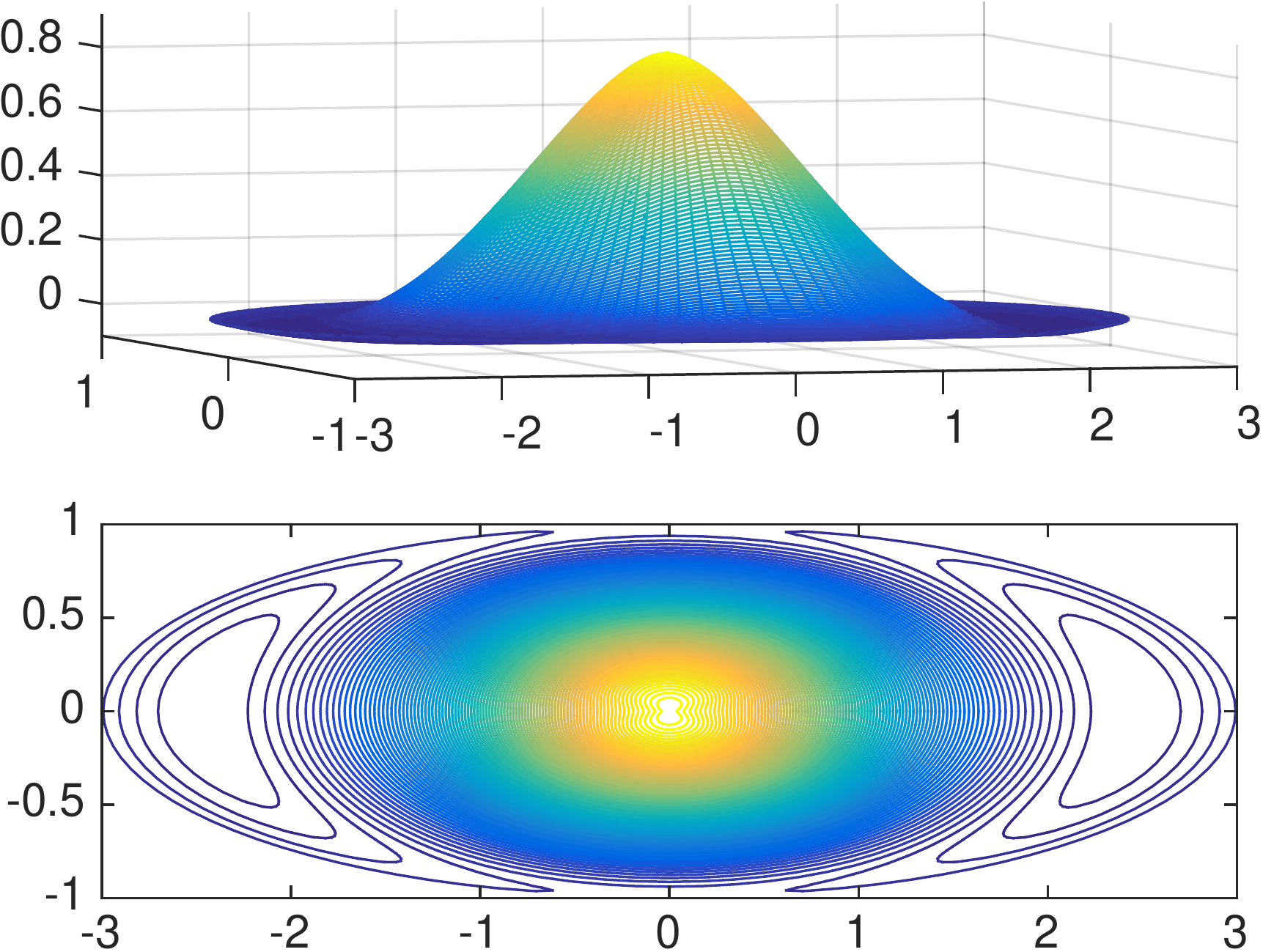}
\caption{Mesh and contour of $\psi_1$.}\label{fig:psi1}
\end{minipage}\hfill%
\begin{minipage}[c]{0.48\textwidth}
\centering
\includegraphics[width=1\textwidth]{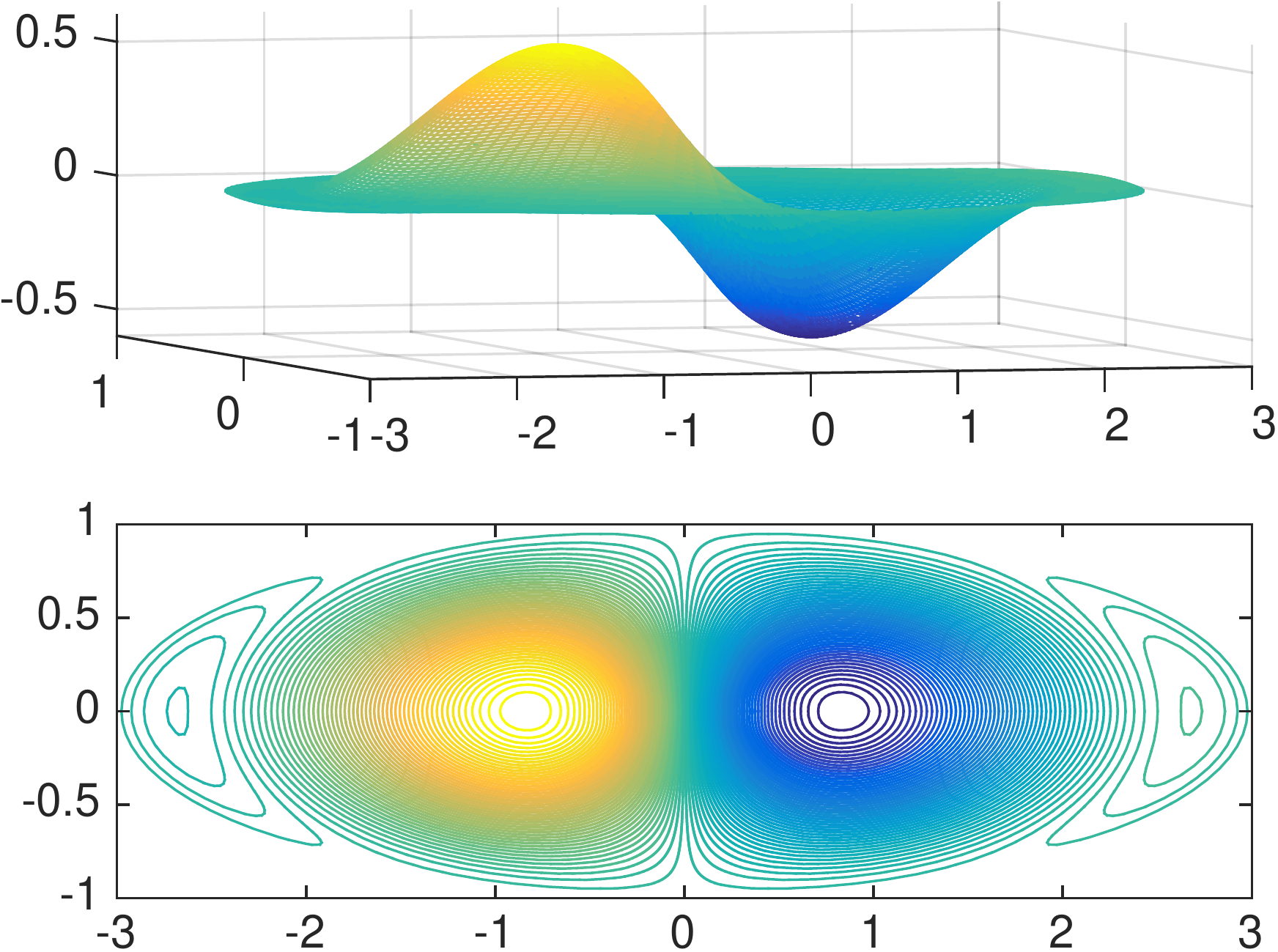}
\caption{Mesh and contour of $\psi_2$.}\label{fig:psi2}
\end{minipage}\hfill%
\end{figure}
\begin{figure}[h!]
\begin{minipage}[c]{0.48\textwidth}
\centering
\includegraphics[width=1\textwidth]{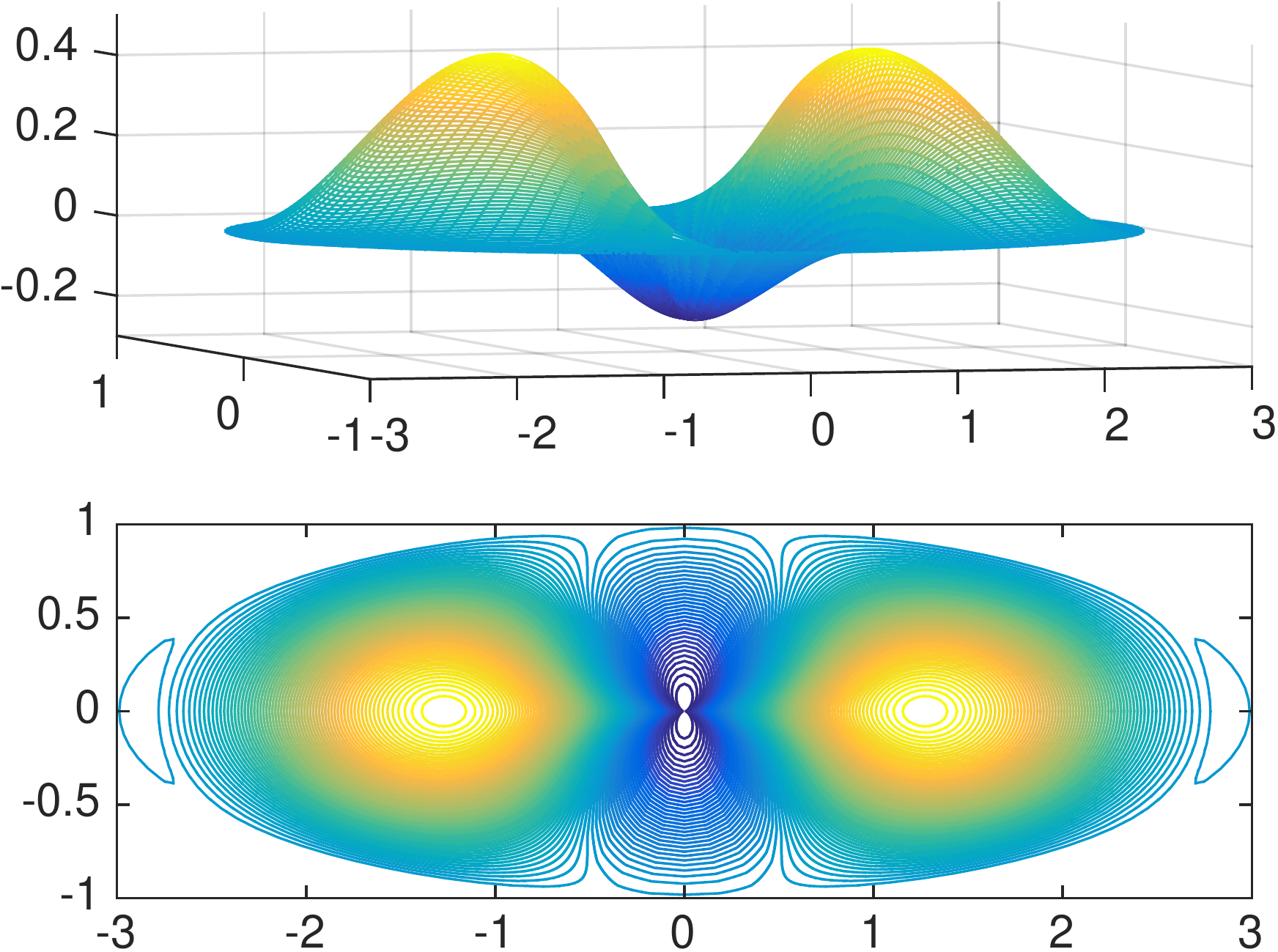}
\caption{Mesh and contour of $\psi_3$.}\label{fig:psi3}
\end{minipage}\hfill%
\begin{minipage}[c]{0.48\textwidth}
\centering
\includegraphics[width=1\textwidth]{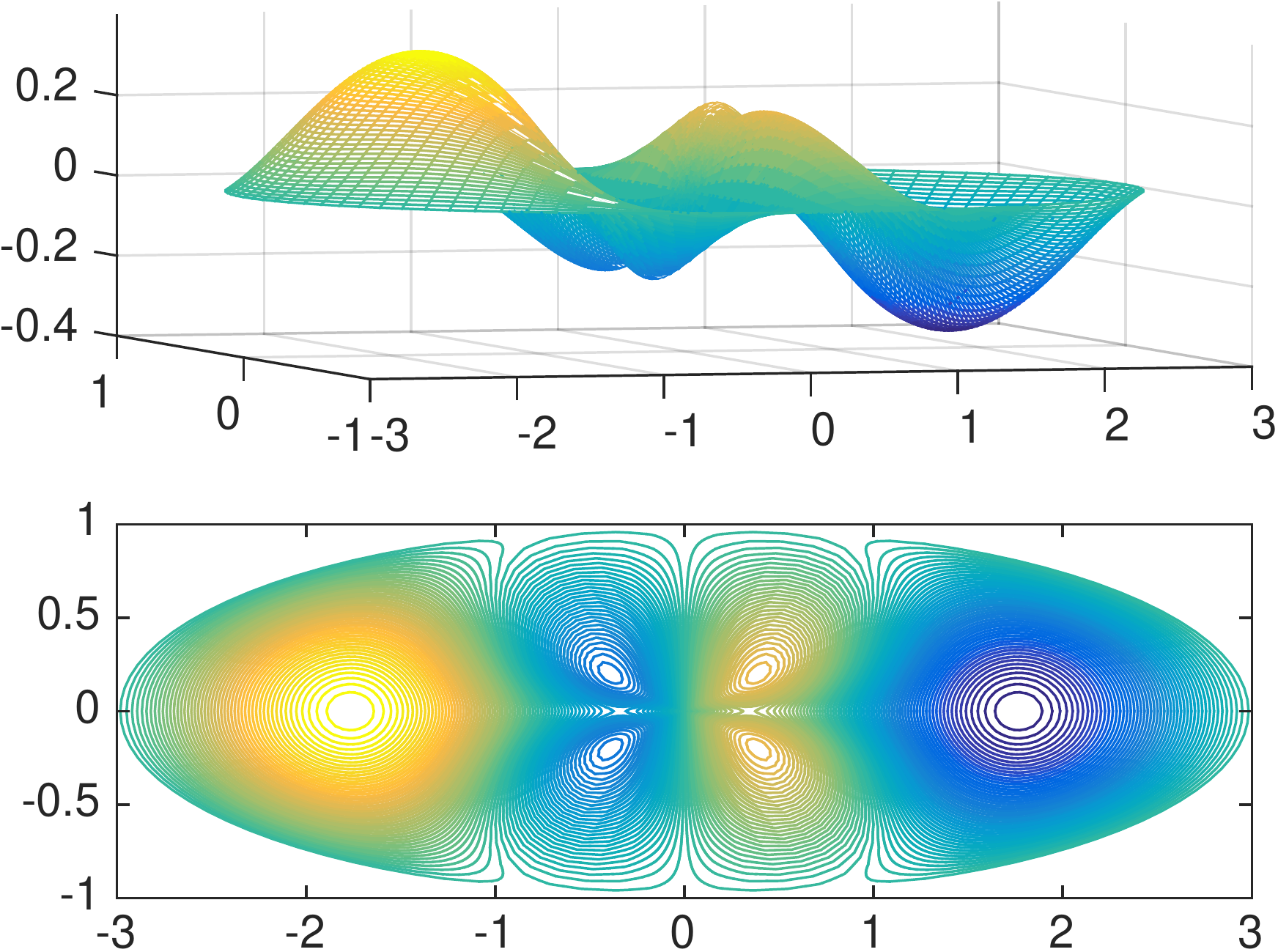}
\caption{Mesh and contour of $\psi_4$.}\label{fig:psi4}
\end{minipage}\hfill%
\end{figure}
\begin{figure}[h!]
\begin{minipage}[c]{0.48\textwidth}
\centering
\includegraphics[width=1\textwidth]{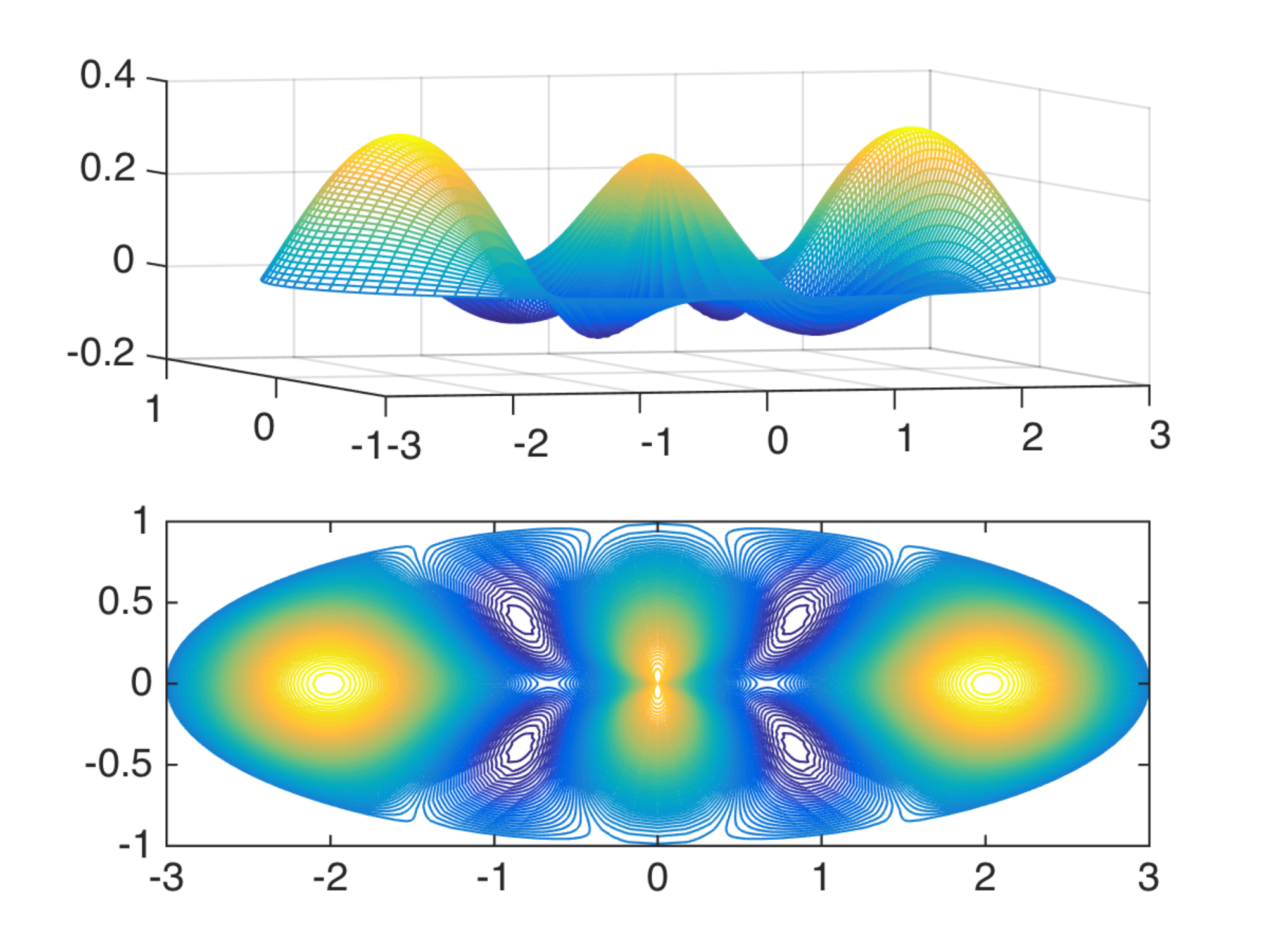}
\caption{Mesh and contour of $\psi_5$.}\label{fig:psi5}
\end{minipage}\hfill%
\begin{minipage}[c]{0.48\textwidth}
\centering
\includegraphics[width=1\textwidth]{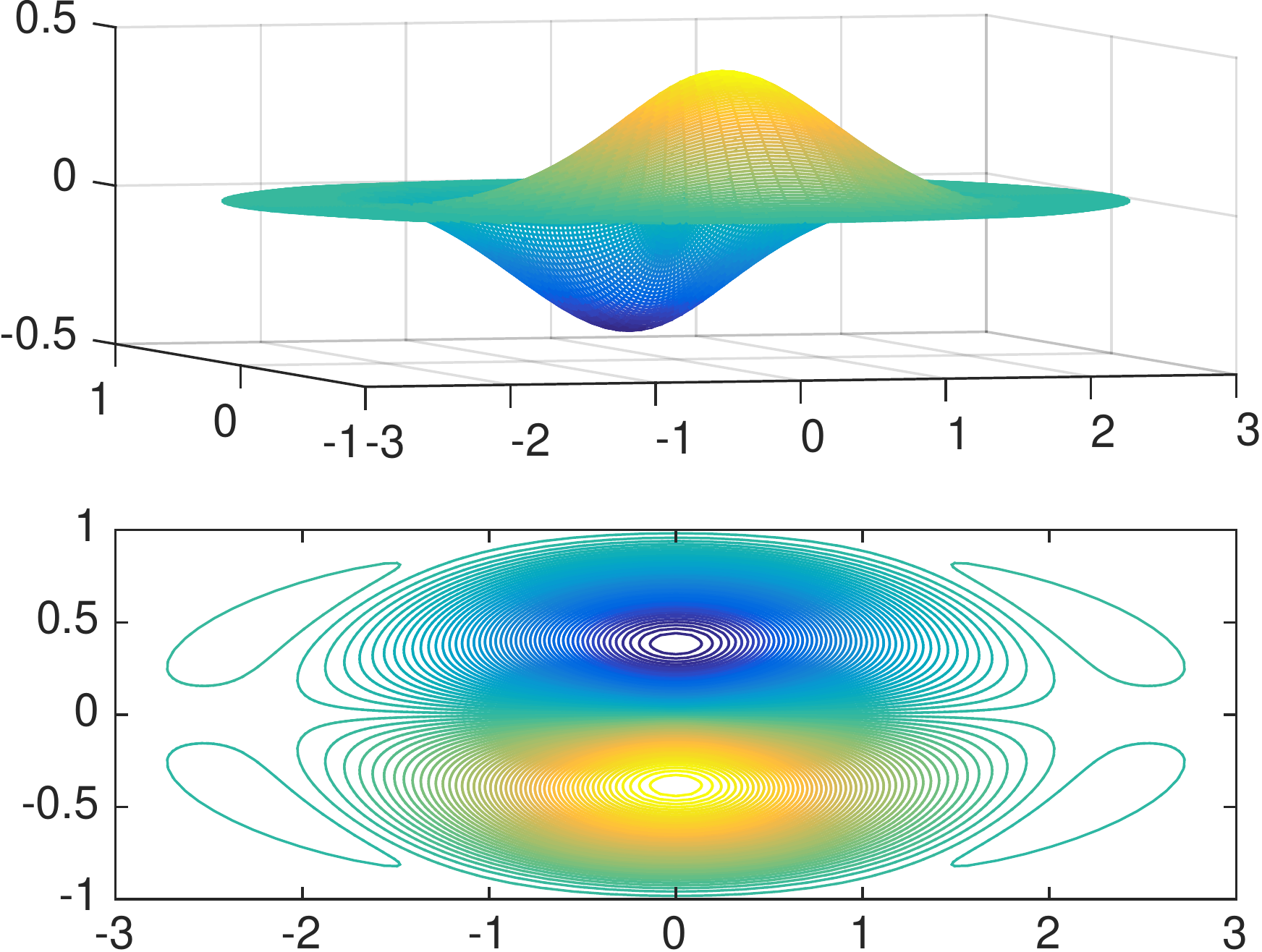}
\caption{Mesh and contour of $\psi_6$.}\label{fig:psi6}
\end{minipage}\hfill%
\end{figure}
\begin{figure}[h!]
\begin{minipage}[c]{0.48\textwidth}
\centering
\includegraphics[width=1\textwidth]{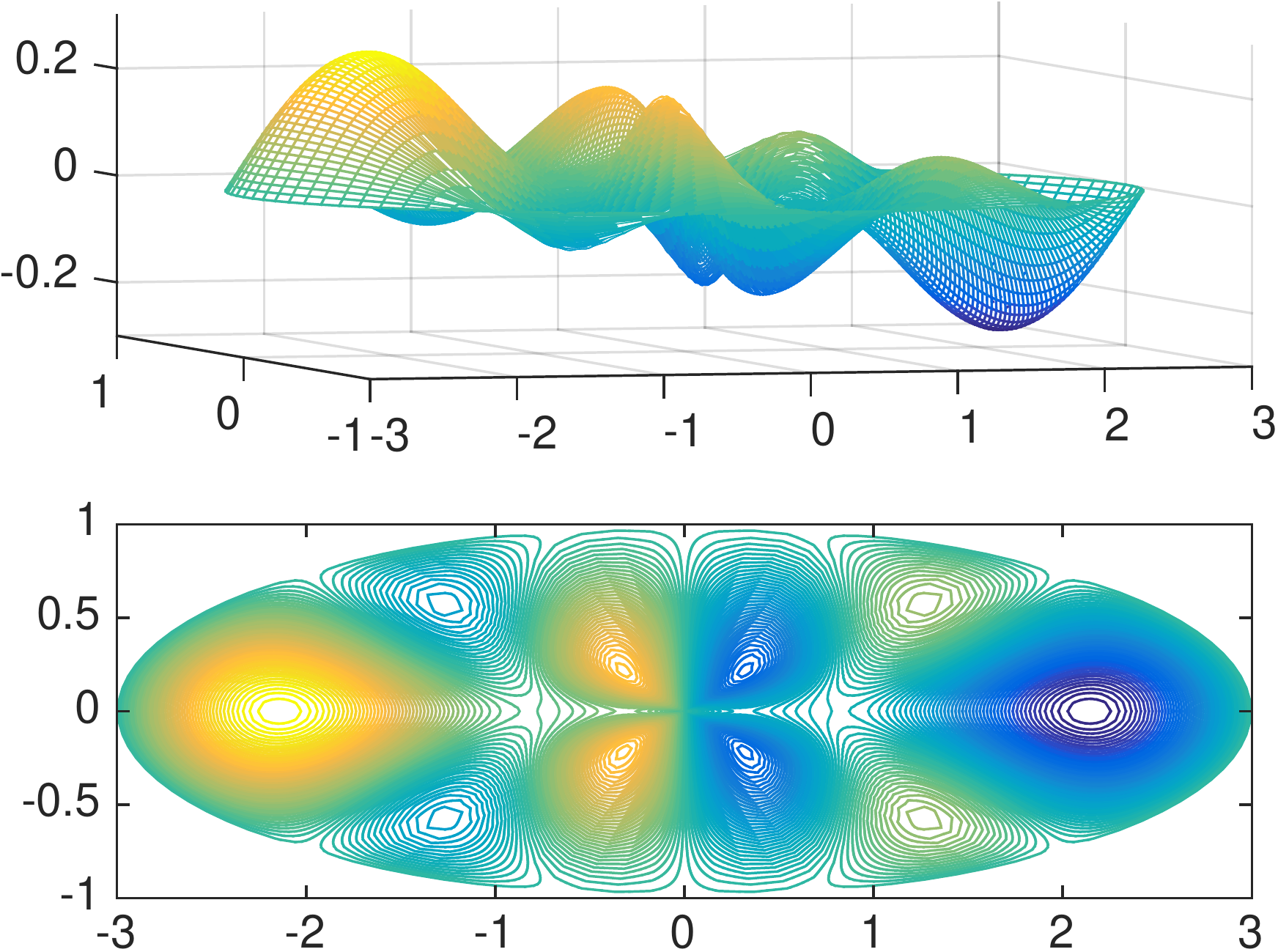}
\caption{Mesh and contour of $\psi_7$.}\label{fig:psi7}
\end{minipage}\hfill%
\begin{minipage}[c]{0.48\textwidth}
\centering
\includegraphics[width=1\textwidth]{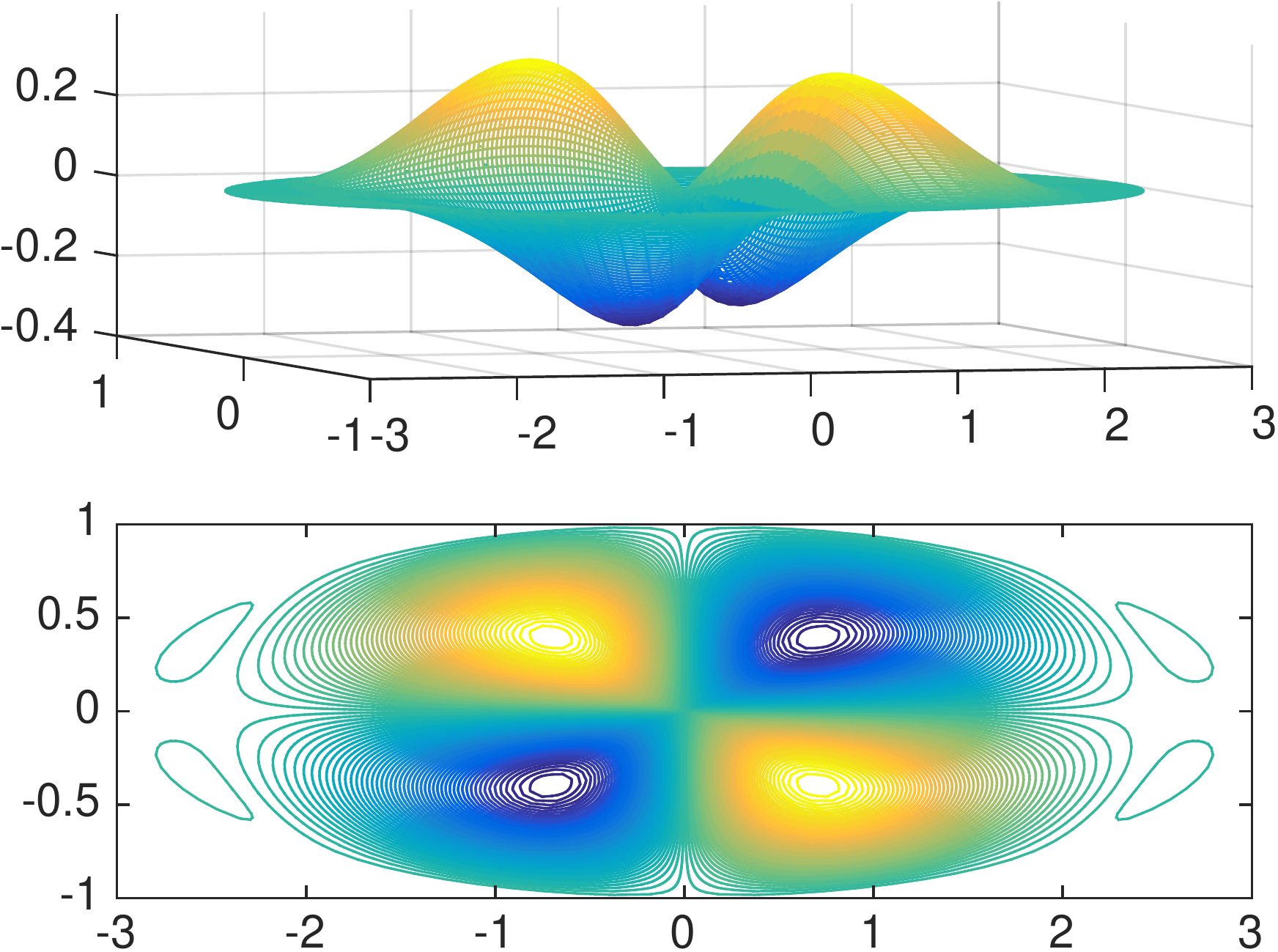}
\caption{Mesh and contour of $\psi_8$.}\label{fig:psi8}
\end{minipage}\hfill%
\end{figure}

\section{Conclusions}\label{conc}
We present a rigorous error analysis { for our proposed spectral-Galerkin methods in solving} the Stokes eigenvalue problem {under} the stream function formulation in polar geometries. We derive the essential pole condition and reduce the problem to
 a sequence of one-dimensional eigenvalue problems that can be solved individually in parallel. 
 { Spectral accuracy is achieved by properly designed non-polynomial basis functions and the exponential rate of convergence is established  by introducing a suitable weighted Sobolev space; all based on the correct pole condition. To the best of out knowledge, the pole condition and such kind of usage of weighted Sobolev space and basis functions are all for the first time in the literature.}
 Our spectral-Galerkin method is {also} extended to {solve} the stream function formulation
 of {the Stokes eigenvalue problem on an elliptic} region, which also indicates  the capability of our method to solve { fourth-order} equations
 on other smooth domains. Numerical experiments in the last section have validated the theoretical results and algorithms.
{As we can see, on special domains such as circular disks and elliptic regions, with only less than 50 degrees of unknowns, the proposed spectral method can achieve 14-digits accuracy for the first few eigenvalues of the Stokes problem, this is far more superior to traditional methods such as finite element and finite difference methods. }

%This work is
% supported by NSFC (Grant Nos. 11201161, 11171125, 91130003).

\appendix 
\section{Jacobi and generalized Jacobi polynomials}
\label{App:A}
The classical Jacobi polynomials $J_k^{\alpha,\beta}(\zeta)$, $ k \ge 0$  with $\alpha,\beta> -1$ are
mutually orthogonal with respect to the Jacobi weight function $\chi^{\alpha,\beta}:=\chi^{\alpha,\beta}
(\zeta) = (1-\zeta)^{\alpha} (1+\zeta)^{\beta}$ on $\Lambda=(-1,1)$,
\begin{align}
\label{Jorth}
        \int_{-1}^{1} J_m^{\alpha,\beta}(\zeta) J_n^{\alpha,\beta}(\zeta) \chi^{\alpha,\beta}(\zeta)
        d\zeta = \frac{2^{\alpha+\beta+1}}
  {{2n+\alpha+\beta+1}}\, h^{\alpha,\beta}_{n}\,  \delta_{m,n}, \quad
        m,n\ge 0,
\end{align}
where $\delta_{m,n}$ is the Kronecker delta, and
\begin{align}
\label{Jnorm}
 h^{\alpha,\beta}_{n} := \frac{\Gamma(n+\alpha+1)\Gamma(n+\beta+1)}
  {\Gamma(n+1)\Gamma(n+\alpha+\beta+1)}.
\end{align}
For $k\in \ZZ$, denote  by $(a)_k=\frac{\Gamma(a+k)}{\Gamma(a)}$
the Pochhammer symbol.
The classical Jacobi polynomials possess  the following important representation
\begin{align}
\label{Jexpand}
 &{J}^{\alpha,\beta}_n(\zeta) =%(-1)^n\binom{n+\beta}{n} {}_2F_1\Big(-n,n+\alpha+\beta+1;\beta+1;\frac{1+x}{2}\Big)=
  \sum_{k=0}^n \frac{(-n-\beta)_{n-k}   (n+\alpha+\beta+1)_k  }{(n-k)!k!} 
  \Big(\frac{\zeta+1}{2}\Big)^k,
\end{align}
which  symbolically furnishes the extension  of $J^{\alpha,\beta}_n(\zeta)$ to arbitrary  $\alpha$ and $\beta$.
Generalized Jacobi polynomials preserve most of the essential  properties of the 
classic Jacobi polynomials, among which the following identities \cite{Szego} are of importance in the current paper,
\begin{align}
\label{symm}
&J^{\alpha,\beta}_n(-\zeta)  =(-1)^n J^{\beta,\alpha}_n(\zeta),
\\
\label{diff}
&\partial_{\zeta} J^{\alpha,\beta}_n(\zeta) = \frac{n+\alpha+\beta+1}{2}
J^{\alpha+1,\beta+1}_{n-1}(\zeta),
%\\
%\label{SL}
% &-  \chi^{-\alpha,\beta}(\zeta) \partial_{\zeta} \big(  \chi^{\alpha+1,\beta+1}(\zeta) \partial_{\zeta}J^{\alpha,\beta}_n(\zeta) \big)   = 
% n(n+\alpha+\beta+1) J^{\alpha,\beta}_n(\zeta) . 
\\
\label{Jexpand2}
&J^{\alpha,\beta}_n(\zeta) %= \frac{\Gamma(n+\beta+1)}{\Gamma(n+\alpha+\beta+1)}
%\\
%&\qquad \times
%\sum_{\nu=0}^{n} \frac{2\nu+\alpha+\beta+k+1}{(n+\nu+\alpha+\beta+1)_{k+1}}
%\frac{\Gamma(\nu+\alpha+\beta+k+1)}{\Gamma(\nu+\beta+1)} \frac{(-k)_{n-\nu}}{(n-\nu)!}J^{\alpha+k,\beta}_{\nu}(\zeta)
%\\
%&
 = \sum_{\nu=n-k}^{n} \frac{2\nu+\alpha+\beta+k+1}{(n+\nu+\alpha+\beta+1)_{k+1}}
\frac{(\nu+\beta+1)_{n-\nu}}{(\nu+\alpha+\beta+k+1)_{n-\nu-k}} \frac{(-k)_{n-\nu}}{(n-\nu)!}J^{\alpha+k,\beta}_{\nu}(\zeta),
\\
\label{rJ}
& (1+\zeta)J^{\alpha,\beta+1}_n(\zeta) = \frac{2(n+\beta+1)}{2n+\alpha+\beta+2} J^{\alpha,\beta}_{n}(\zeta)
+  \frac{2(n+1)}{2n+\alpha+\beta+2} J^{\alpha,\beta}_{n+1}(\zeta).
\end{align}

In particular, the generalized Jacobi polynomials with $\alpha$ and/or $\beta$ being integers are our greatest interest
\cite{LS09},
\begin{align}
  \label{eq:gjp}
  \begin{split}
  J_n^{\alpha,\beta}(\zeta)
  &=\begin{cases}   \big(\frac{\zeta-1}{2}\big)^{-\alpha}  \big(\frac{\zeta+1}{2}\big)^{-\beta}
   J_{n+\alpha+\beta}^{-\alpha, -\beta}(\zeta), & \alpha,\beta\in \ZZ,\ 
   n+\alpha+\beta\in \NN_0,\\[0.4em]
   h^{\alpha,\beta}_n
   \big(\frac{\zeta-1}{2}\big)^{-\alpha}
   J_{n+\alpha}^{-\alpha, \beta}(\zeta), & \alpha\in \ZZ,\ 
   n+\alpha\in \NN_0, \\[0.4em]
   h^{\alpha,\beta}_n
   \big(\frac{\zeta+1}{2}\big)^{-\beta}
   J_{n+\beta}^{\alpha, -\beta}(\zeta), & \beta\in \ZZ, \
    n+\beta\in \NN_0.
  \end{cases}
  \end{split}
\end{align}
The generalized Jacobi polynomials with
negative indices not only simplify the numerical analysis for the
spectral approximations of differential equations, but also lead to
very efficient numerical algorithms \cite{GSW06,shen2011}.

Finally, it is worthy to point out that a reduction of the degree of $J^{\alpha,\beta}_n(\zeta)$ occurs if and only if   $-n-\alpha-\beta\in \{1,2,\dots,n\} $,
\begin{align}
   \label{Jc}
      &J^{\alpha,\beta}_n(\zeta) =  
            h^{\alpha,n_0-n-1}_n
          %(-1)^{n-n_0+1}h^{n_0-n-1,\beta}_n 
             J^{\alpha,\beta}_{n_0-1}(\zeta),
\end{align}
where $n_0:=-n-\alpha-\beta$ if $-n-\alpha-\beta\in \{1,2,\dots,n\} $ and $n_0:=0$ otherwise.

\section{Proof of Lemma \ref{Expan}}
\label{App:B}
At first,  \eqref{phi3d2}-\eqref{phi31} are trivial results on the Jacobi expansion.

By \eqref{symm}, \eqref{Jexpand2}, \eqref{eq:gjp}, one finds that, for $i\ge 4$,
\begin{align*}
&(1-t^2)^2 J^{2,1}_{i-4}(t) 
=   \frac12 (1-t^2)^2 \big[ \frac{i}{i-2} J^{2,2}_{i-4}(t)+J^{2,2}_{i-5}(t)  \big]
=  \frac{8  i}{i-2} J^{-2,-2}_{i}(t)+8(1-\delta_{i,4}) J^{-2,-2}_{i-1}(t).
\end{align*}
Then by \eqref{diff}, \eqref{symm} and \eqref{Jexpand2}, one derives
\begin{align*}
&\partial_t^2 [(1-t^2)^2 J^{2,1}_{i-4}(t) ] = 2i(i-3)J^{0,0}_{i-2}(t)+2(i-4)(i-3)J^{0,0}_{i-3}(t)
\\
&\quad =\frac{2i(i-3)}{2i-3} ((i-1)J^{0,1}_{i-2}(t)+(i-2) J^{0,1}_{i-3}(t) ) +\frac{2(i-4)(i-3)}{2i-5}
 ((i-2)J^{0,1}_{i-3}(t)+(i-3)J^{0,1}_{i-4}(t) ) 
\\
&\quad = {\frac {2 ( i-3 )  ( i-1 ) i}{(2i-3)}}  J^{0,1}_{i-2}(t) 
 +   {\frac {8 ( i-3 ) ^{2} ( i-2 )  ( i-1
 ) }{ ( 2i-3 )  ( 2i-5 ) }}  J^{0,1}_{i-3}(t)
+ {\frac { 2( i-4 )  ( i-3 ) ^{2}}{(2i-5)}}J^{0,1}_{i-4}(t)
,
\end{align*}
and
\begin{align*}
&\frac{1}{t+1}\partial_t [(1-t^2)^2 J^{2,1}_{i-4}(t) ] =\frac{1}{t+1}\big[ \frac{4  i (i-3)  }{i-2} J^{-1,-1}_{i-1}(t)+4(i-4) J^{-1,-1}_{i-2}(t) \big]
\\
&\quad =\big[ \frac{2  i (i-3)  }{i-2} J^{-1,1}_{i-2}(t)+\frac{2(i-4)(i-3)}{i-2} J^{-1,1}_{i-3}(t) \big]
\\
&\quad =\frac { 2( i-3 ) i}{(2i-3)} J^{0,1}_{i-2}(t) 
 - \frac {12 ( i-3 )  ( i-2 ) }{ ( 2i-3 )  ( 2i-5 ) } J^{0,1}_{i-3}(t)
- \frac {2 ( i-4 )  ( i-3 ) }{(2i-5)}J^{0,1}_{i-4}(t),
\end{align*}
which give \eqref{phid2} and \eqref{phid1} immediately.

Next by \eqref{eq:gjp} and \eqref{Jexpand2},
\begin{align*}
&\frac{1}{(t+1)^2}(1-t^2)^2 J^{2,1}_{i-4}(t) 
=(1-t)^2 J^{2,1}_{i-4}(t) = \frac{4(i-3)}{i-1}J^{-2,1}_{i-2}(t)
\\
&\qquad = \frac {2(i-3)}{2i-3}J^{0,1}_{i-2}(t) 
-\frac { 8 ( i-3 )  ( i-2 ) }{ ( 2i-3 )  ( 2i-5 ) }J^{0,1}_{i-3}(t)
+\frac {2(i-3)}{2i-5} J^{0,1}_{i-4}(t),
\end{align*}
which states \eqref{phid0}.

Further,  by \eqref{diff}, \eqref{symm} and \eqref{Jexpand2},
\begin{align*}
\partial_t[(1-&t^2)^2 J^{2,1}_{i-4}(t)] =\frac{4  i (i-3)  }{i-2} J^{-1,-1}_{i-1}(t)+4(i-4) J^{-1,-1}_{i-2}(t)
\\
&=\frac{2(i-3)i^2}{(2i-3)(2i-1)}  J^{0,1}_{i-1}(t)
+\frac{4(i-1)(i-3)(2i^2-7i+2)}{(2i-5)(2i-1)(2i-3)} J^{0,1}_{i-2}(t)
\\
&-\frac{8(i-2)(i-3)}{(2i-3)(2i-5)} J^{0,1}_{i-3}(t)
-\frac{4(2i^2-9i+6)(i-3)^2}{(2i-7)(2i-5)(2i-3)} J^{0,1}_{i-4}(t)
-\frac{2(i-3)(i-4)^2}{(2i-5)(2i-7)} J^{0,1}_{i-5}(t),
\end{align*}
and  by \eqref{eq:gjp}, \eqref{symm} and \eqref{Jexpand2},
\begin{align*}
(1-t)^2& (1+t)J^{2,1}_{i-4}(t) =8J^{-2,-1}_{i-1}(t)
= \frac{2i(i-3)}{(2i-1)(2i-3)} J^{0,1}_{i-1}(t)
-\frac{8(i-1)(i-3)}{(2i-1)(2i-3)(2i-5)} J^{0,1}_{i-2}(t)
\\
&-\frac{4(i-2)(i-3)}{(2i-3)(2i-5)} J^{0,1}_{i-3}(t)
+\frac{8(i-3)^2}{(2i-7)(2i-3)(2i-5)} J^{0,1}_{i-4}(t)
+\frac{2(i-3))(i-4)}{(2i-5)(2i-7)} J^{0,1}_{i-5}(t),
\end{align*}
which lead to \eqref{phid1r} and \eqref{phid0r}, respectively.

Finally, \eqref{Aij} and \eqref{Bij} are direct consequences of 
\eqref{phid2}-\eqref{phi31} and \eqref{Jorth}. \endproof

\end{document}